\documentclass[12pt]{amsart}
\usepackage{epsfig}
\usepackage{amsfonts}
\usepackage{color}
\usepackage{graphicx}
\usepackage{amsmath}
\usepackage{amssymb}
\newtheorem{thm}{Theorem}[section]

\newtheorem{cor}[thm]{Corollary}

\newtheorem{prop}[thm]{Proposition}
\newtheorem*{prob*}{Problem}
\newtheorem*{thm*}{Theorem}

\theoremstyle{definition}
\newtheorem{defn}[thm]{Definition}
\newtheorem{example}[thm]{Example}
\newtheorem*{defn*}{Definition}
\newtheorem{rem}[thm]{Remark}
\newtheorem*{Remarks}{Remarks}
\newtheorem*{rem*}{Remark}
\numberwithin{equation}{section}

\DeclareMathOperator{\Tr}{Tr}

\DeclareMathOperator{\Ind}{Ind}
\DeclareMathOperator{\SYM}{SYM}

\DeclareMathOperator{\Res}{Res}

\DeclareMathOperator{\R}{\mathbb{R}}

\DeclareMathOperator{\Y}{\mathbb{Y}}
\DeclareMathOperator{\C}{\mathbb{C}}

\DeclareMathOperator{\diag}{diag}

\DeclareMathOperator{\DIM}{DIM}

\DeclareMathOperator{\Ewens}{Ewens}

\DeclareMathOperator{\Reg}{Reg}

\begin{document}
\title[Probability measures on families of partitions]
{\bf{Probability measures on families of partitions related to harmonic analysis on big wreath products}}

\author{Eugene Strahov}
\address{Department of Mathematics, The Hebrew University of Jerusalem, Givat Ram, Jerusalem 91904, Israel}
\email{strahov@math.huji.ac.il}
\commby{}
\begin{abstract}
We construct generalized regular representations of the wreath product of a compact group with the infinite symmetric group. The characters of these representations are determined by probability measures on families
of partitions called the $z$-measures for the wreath product of a compact group with the symmetric group in the present paper. Our main result is an explicit formula for these $z$-measures which holds true for an arbitrary compact group. The result enables us to describe the spectral measures of the generalized regular representations of big wreath products.
\end{abstract}
\maketitle
\section{Introduction}
\subsection{Formulation of the problem and description of main results}
Let $S_{\infty}$ be the infinite symmetric group, that is, the group of finite permutations
of $\{1,2,\ldots\}$.
The representation theory of $S_{\infty}$ and other big groups is an active subject of current research, with many connections to different areas of mathematics.
We refer the reader to the book by Borodin and Olshanski \cite{BorodinOlshanskiBook}, to the book by Kerov \cite{KerovDissertation}, and to the lectures by Hora \cite{HoraLectures} for an introduction
to the representation theory of $S_{\infty}$ and descriptions of some of its applications.

The harmonic analysis on $S_{\infty}$ is developed by Kerov, Olshanski and Vershik \cite{KerovOlshanskiVershikAnnouncement, KerovOlshanskiVershik}, and by Borodin and Olshanski;
see Refs. \cite{Borodin1, Borodin2, BorodinOlshanskiLetters}, and the survey papers by Olshanski \cite{OlshanskiPointProcesses, OlshanskiNato}.  The theory is based on the construction
of generalized regular representations of the infinite symmetric group, and on the description
of the characters of these representations in terms of certain probability measures on partitions
called the $z$-measures for the symmetric group.

A natural direction of research is to extend the results on the representation theory
of the infinite symmetric group to other big groups. An example is the infinite
unitary group; see Olshanski \cite{OlshanskiUnitary}, Borodin, and Olshanski \cite{BorodinOlshanskiUnitary} for different results in this case.
The  papers by  Gorin, Kerov, and  Vershik \cite{GorinKerovVershik}, Cuenca and Olshanski \cite{CuencaOlshanski1}
are devoted to characters and  representations of the group of infinite matrices over a finite field.

Let $G$ be a compact group.  Denote by $S_{\infty}(G)$ the wreath product of $G$ with $S_{\infty}$.
The group $S_{\infty}(G)$ called the big wreath product is an important generalization of the infinite symmetric group $S_{\infty}$.
The works of Hora, Hirai and Hirai
\cite{HoraHiraiHiraiII}, Hirai, Hirai, and Hora \cite{HiraiHiraiHoraI},
Hora and Hirai \cite{HoraHirai} describe the characters of $S_{\infty}(G)$.
The purpose of the present paper is to extend some results of the harmonic analysis of $S_{\infty}$ to $S_{\infty}(G)$. That is, we will construct generalized regular representations
of $S_{\infty}(G)$, and we will derive formulas for the characters of these representations.
We will see that these characters are determined by probability measures on families of partitions.
For these probability measures (called the $z$-measures for the wreath product $S_n(G)$ of a compact group $G$ with the symmetric group $S_n$), we will derive explicit formulae. Then we will use these formulae to describe the spectral measures of the generalized regular representations.
In the case of a finite group $G$, our results are reduced to those obtained in Strahov \cite{StrahovIsr25}.

On the technical level, there is an essential difference between the case where
$G$ is a compact group with an infinite number of elements (like the unitary group $U(n)$) and the case where $G$ is a finite group (considered by Strahov \cite{StrahovIsr25}). Indeed, in the case of a finite $G$
there is an analogue of the Frobenius theory which relates the characters of $S_n(G)$
with the theory of symmetric functions; see Macdonald \cite{Macdonald}, Appendix B.
This relation was used in Strahov \cite{StrahovIsr25} to derive the formulae
for the $z$-measures. If $G$ is not a finite group, an analogue of Frobenius theory is absent for
the characters of $S_n(G)$. To overcome this difficulty, we will use the construction of irreducible representations and formulas for the characters of $S_n(G)$ from the works of Hirai, Hirai, Hora \cite{HiraiHiraiHoraI},
Hora, Hirai, Hirai \cite{HoraHiraiHiraiII} and Hora and Hirai \cite{HoraHirai}.

In the following, we give a short description of the main results obtained in this paper.
\subsubsection{The Ewens measure on a wreath product of a compact group with the symmetric group} There is a natural one-parameter family of distributions $\left(P^{\Ewens}_{n,\theta}\right)_{\theta>0}$ on the symmetric group $S_n$ defined by
\begin{equation}\label{Res.1}
P^{\Ewens}_{n,\theta}(s)=\frac{\theta^{l(s)}}{\theta(\theta+1)\ldots(\theta+n-1)},\; s\in S_n,
\end{equation}
where $l(s)$ denotes the number of cycles in $s$. The distribution
$P^{\Ewens}_{n,\theta}$ is a deformation of the uniform distribution on $S_n$, and it is
called the Ewens measure on $S_n$ with the parameter $\theta>0$.
It is not hard to see that $P_{n,\theta}^{\Ewens}$ is invariant under the action
of $S_n$ on itself by conjugations. Moreover, the Ewens measure is consistent with respect to the natural projection $p_{n,n-1}:\; S_n\longrightarrow S_{n-1}$.
These fundamental properties of $P_{n,\theta}^{\Ewens}$ lead to applications in the representation theory of the infinite symmetric group, and in population genetics.
In particular, $P_{n,\theta}^{\Ewens}$ is closely related to the Ewens sampling formula; see, for example, Kingman \cite{Kingman1,Kingman2}, Tavar$\acute{\mbox{e}}$ \cite{Tavare}.

In the present paper, we introduce the Ewens-type measure $\mu_{S_n(G)}^{\Ewens}$ on the wreath product $S_n(G)$, where $G$ is a compact group, see Definition \ref{Definition2.1}.
The measure $\mu_{S_n(G)}^{\Ewens}$ is determined by a \textit{continuous central function} $z:\; G\longrightarrow\C\setminus\{0\}$, and is a probability measure on $S_n(G)$ invariant under the action of $S_n(G)$ on itself by conjugations.  We define the canonical projection $p_{n,n+1}:\; S_{n+1}(G)\longrightarrow S_n(G)$ that preserves the colors of the cycles.
The fundamental property of $\mu_{S_n(G)}^{\Ewens}$ is its consistency with $p_{n,n+1}$.
If $G$ contains the unit element only, then $\mu_{S_n(G)}^{\Ewens}$
turns into the Ewens measure on the symmetric group $S_n$. If $G$ is a finite group, then
$\mu_{S_n(G)}^{\Ewens}$ leads to a refinement of the Ewens sampling formula described in Strahov \cite{StrahovRefinement}. In this paper, we use $\mu_{S_n(G)}^{\Ewens}$
to construct generalized regular representations of the wreath product $S_{\infty}(G)$
of a compact group $G$ with the infinite symmetric group $S_{\infty}$.
\subsubsection{The generalized regular representations $T_z$ of $S_{\infty}(G)$.}
We introduce the space $\mathfrak{S}_G$ of $G$-virtual permutations as a projective limit
$$
S_1(G)\longleftarrow S_2(G)\longleftarrow\ldots\longleftarrow S_n(G)\longleftarrow\ldots
$$
under the fundamental projection $p_{n,n+1}$. Since the family $\left(\mu^{\Ewens}_{S_n(G)}\right)_{n=1}^{\infty}$ is consistent with $p_{n,n+1}$, we obtain a probability space
$\left(\mathfrak{S}_G, \mu_{\mathfrak{S}_G}^{\Ewens}\right)$, where
$\mathfrak{S}_G=\underset{\longleftarrow}{\lim}\;S_n(G)$, and $\mu_{\mathfrak{S}_G}^{\Ewens}
=\underset{\longleftarrow}{\lim}\;\mu^{\Ewens}_{S_n(G)}$. Then we define the action of $S_{\infty}(G)\times S_{\infty}(G)$ on $\mathfrak{S}_G$, and construct a family
$\left(T_z,L^2\left(\mathfrak{S}_G,\mu_{\mathfrak{S}_G}^{\Ewens}\right)\right)$
of representations of $S_{\infty}(G)\times S_{\infty}(G)$ parameterized by the continuous central function $z:\; G\longrightarrow \C\setminus\{0\}$, see Definition \ref{12.2.1}.
If $G$ contains the unit element only, the representations
$\left(T_z,L^2\left(\mathfrak{S}_G,\mu_{\mathfrak{S}_G}^{\Ewens}\right)\right)$
are equivalent to the generalized regular representations of the infinite symmetric group. In the case where $G$ is a finite group, the representations $\left(T_z,L^2\left(\mathfrak{S}_G,\mu_{\mathfrak{S}_G}^{\Ewens}\right)\right)$ were constructed in Strahov \cite{StrahovIsr25}.
\subsubsection{A formula for the characters of $T_z$. The $z$-measures for the wreath products $S_{n}(G)$}
The character $\chi_z$ of $T_z$ can be defined as a spherical function of a suitable spherical representation of the Gelfand pair
$\left(S_{\infty}(G)\times S_{\infty}(G),\diag\left(S_{\infty}(G)\right)\right)$, see Section \ref{Section12.4}. In the paper, we give a formula for $\chi_z$ assuming that the set $\widehat{G}$
of equivalence classes of continuous unitary irreducible representations of $G$ is at most countable.
\begin{defn}\label{DefinitionZMES} Let $z: G\longrightarrow\C\setminus\{0\}$ be a continuous central function, and let
$$
\left(T_z,L^2\left(\mathfrak{S}_G,\mu^{\Ewens}_{\mathfrak{S}_G}\right)\right)
$$
be the generalized regular representation of $S_{\infty}(G)\times S_{\infty}(G)$.
Denote by $\chi_z$ the character of $T_z$ as defined in Section
\ref{Section12.4}. Consider the expansion of $\chi_z\vert_{S_n(G)}$
into the sum of normalized irreducible characters,
\begin{equation}\label{13.1.2}
\chi_z\vert_{S_n(G)}(x)=\sum\limits_{\Lambda\in\Y_n(\widehat{G})}
M_z^{(n)}\left(\Lambda\right)\frac{\chi^{\Lambda}(x)}{\DIM\Lambda},
\end{equation}
where $\Y_n\left(\widehat{G}\right)$ is the set of all families of Young diagrams parameterizing
nonequivalent classes of irreducible representations of $S_n(G)$,
and $\DIM\Lambda$ is the dimension of the irreducible representation of $S_n(G)$ parameterized by $\Lambda$. The coefficient $M_z^{(n)}(\Lambda)$ in
this expansion is a probability measure on $\Y_n(\widehat{G})$ determined by the function $z$. This coefficient is called the $z$-measure for the wreath product $S_n(G)$.
\end{defn}
Theorem \ref{Theorem13.2.2} below gives an explicit formula for these $z$-measures.
\begin{thm}\label{Theorem13.2.2}
Assume that an irreducible representation of $S_n(G)$ is parameterized by a
collection $\Lambda$ of Young diagrams
\begin{equation}\label{13.2.2.1}
\Lambda=\left(\Lambda(\zeta)\right)_{\zeta\in\widehat{G}},\;
\sum\limits_{\zeta\in\widehat{G}}|\Lambda(\zeta)|=n.
\end{equation}
Then we have
\begin{equation}\label{13.2.2.2}
M_z^{(n)}\left(\Lambda\right)=\frac{n!}{\left(I\right)_n}
\prod\limits_{\zeta\in\widehat{G}}\;\prod\limits_{\Box\in\Lambda(\zeta)}
\frac{\left(\alpha(\zeta)+c(\Box)\right)\left(\overline{\alpha(\zeta)}+c(\Box)\right)}{h^2(\Box)},
\end{equation}
where $I=\left<z,z\right>_G$, $(I)_n=I(I+1)\ldots (I+n-1)$.
The function $\alpha : \widehat{G}\longrightarrow\C$ is defined by
\begin{equation}\label{13.2.2.3}
\alpha(\zeta)=\left<z,\chi^{\pi^\zeta}\right>_G,
\end{equation}
where $\chi^{\pi^{\zeta}}$ denotes the character of the irreducible representation $\pi^{\zeta}$ of $G$ parameterized by $\zeta\in\widehat{G}$. The inner product  $\left<.,.\right>_G$ is defined by
\begin{equation}
\left<\varphi,\psi\right>_G=\int\limits_G\varphi(g)\overline{\psi(g)}d\mu_G(g),
\end{equation}
where $\mu_G$ is the normalized Haar measure on the group $G$.

For a box $\Box=(i,j)$ of a Young diagram $\lambda=(\lambda_1,\lambda_2,\ldots)$
located in row $i$ and column $j$ we have used notation
$$
c(\Box)=j-i,\;\; h(\Box)=\lambda_i-i+\lambda_j'-j+1,
$$
where $\lambda'$ stands for the transposed diagram.
\end{thm}
\begin{Remarks}
\textbf{(a)} Theorem \ref{Theorem13.2.2} holds for any compact group $G$ such that the set $\widehat{G}$ of equivalence classes of continuous unitary irreducible representations of $G$ is at most countable. It generalizes the formulas for the $z$-measures for the symmetric group, and for the wreath product of a finite group with the infinite symmetric group available in the literature. The $z$-measures  for the symmetric group were studied
by several authors. These measures were first defined in the context of harmonic analysis in
Kerov, Olshanski, and Vershik \cite{KerovOlshanskiVershikAnnouncement, KerovOlshanskiVershik}.
Later it was observed that for special values of parameters $z$-measures for the symmetric group
turn into discrete orthogonal polynomial ensembles, and are to related probabilistic models of statistical mechanics,
see Borodin and Olshanski \cite{BorodinOlshanskiRSK}.  In Refs.\cite{Borodin1, Borodin2, BorodinOlshanskiLetters,
BorodinOlshanskiKernel, BorodinOlshanskiMarkov, BorodinOlshanskiStrahov} the reader can find many results
on the $z$-measures for the symmetric group, and on associated determinantal point processes.
In addition, the $z$-measures for the symmetric group are a particular case of the Schur measure introduced in Okounkov \cite{OkounkovSchur}. Strahov \cite{StrahovMPS} (see equation (8.3)) gives a formula for the $z$-measures associated with the wreath product of a finite group $G$ with the symmetric group $S_n$.\\
(\textbf{b}) The functions $z: G:\longrightarrow\C\setminus{0}$ and
$\alpha: \widehat{G}\longrightarrow\C$ in Theorem \ref{Theorem13.2.2} are related
by the Parseval identity for  compact groups,
\begin{equation}\label{Parseval}
\left<z,z\right>_G=\sum\limits_{\zeta\in\widehat{G}}\alpha(\zeta)\overline{\alpha(\zeta)}.
\end{equation}
(\textbf{c}) If $G=U(1)$, then $\widehat{G}=\mathbb{Z}$,  the irreducible characters
$\chi^{\pi^{l}}$ are functions on the unit circle,
$\chi^{\pi^{l}}(\varphi)=e^{il\varphi}$, and $M_z^{(n)}$ is a probability measure on $S_n(U(1))$. It follows that
\begin{equation}\label{Parseval1}
\sum\limits_{\Lambda\in\Y_n(\mathbb{Z})}
\prod\limits_{l=-\infty}^{+\infty}\prod\limits_{\Box\in\lambda^{(l)}}
\frac{(\alpha(l)+c(\Box))(\overline{\alpha(l)}+c(\Box))}{h^2(\Box)}
=\frac{1}{n!}\left(\frac{1}{2\pi}\int\limits_0^{2\pi}|z(e^{i\theta})|^2d\theta\right)_n,
\end{equation}
where the sum is over the set $\Y_n(\mathbb{Z})$ of all families $\Lambda=\left(\lambda^{(l)}\right)_{l=-\infty}^{+\infty}$ of Young diagrams such that
$$
\sum\limits_{l=-\infty}^{+\infty}\left|\lambda^{(l)}\right|=n,
$$
and where
$$
\alpha(l)=\frac{1}{2\pi}\int\limits_0^{2\pi}z(e^{i\varphi})e^{-il\varphi}d\varphi, \;\;l\in\mathbb{Z},
$$
are the Fourier coefficients of the function $z$ on the unit circle.
As $n=1$, equation (\ref{Parseval1})  is reduced to
\begin{equation}
\sum\limits_{l=-\infty}^{+\infty}\alpha(l)\overline{\alpha(l)}=
\frac{1}{2\pi}\int\limits_0^{2\pi}\left|z(e^{i\varphi})\right|^2d\varphi,
\end{equation}
which is
the usual Parseval identity for the Fourier coefficients $(\alpha(l))_{l=-\infty}^{+\infty}$ of the function $z$ on the unit circle.
\end{Remarks}

\subsubsection{Description of the spectral $z$-measures}
An analogue of Thoma's theorem \cite{Thoma} for the big wreath product $S_{\infty}(G)$ of a compact group $G$ with the infinite symmetric  group $S_{\infty}$ can be found in Hora and Hirai \cite{HoraHirai}. According to this theorem the character $\chi_z$ of the generalized  regular representation $\left(T_z,L^2\left(\mathfrak{S}_G,\mu_{\mathfrak{S}_G}^{\Ewens}\right)\right)$
can be written as
\begin{equation}\label{ZcharacterRepresentation}
\chi_z(\varrho)=\int\limits_{\triangle}f_{\omega}(\varrho)dP_z(\omega),
\end{equation}
where $\triangle$ is a generalized Thoma set, $\chi_z(\varrho)$ is the value of $\chi_z$ at an element of the conjugacy class of $S_{\infty}(G)$
parameterized by the family $\varrho$ of partitions,  $P_z$ is a probability measure on $\triangle$, and $f_{\omega}(\varrho)$ is a kernel. The formulae for $f_{\omega}(\varrho)$ and $\triangle$ are given in Hora and Hirai \cite{HoraHirai}, see Theorem 3.4. The problem of harmonic analysis is to describe the measure $P_z$. In representation-theoretical terms, this is equivalent to the description of the decomposition of $T_z$ into irreducible components.
Our Theorem \ref{THEOREMSPECTRAL} gives the description of $P_z$ in terms of the spectral $z$-measures associated with the generalized regular representations of the infinite symmetric group. The spectral $z$-measures for the infinite symmetric group were studied in detail by Borodin and Olshanski in Refs. \cite{Borodin1, Borodin2, BorodinOlshanskiLetters}.
\subsection{Organization of the paper}
The paper is organized as follows. In \textbf{Section \ref{SECTIONNOTATION}} we introduce the notation for
the wreath product $S_n(G)$ of a compact group $G$ with the symmetric group $S_n$, and review the basic facts on conjugacy classes and irreducible representations of $S_n(G)$.
In \textbf{Section \ref{SECTIONEWENSMEASURE}} we define the Ewens measure $\mu_{S_n(G)}^{\Ewens}$ on the wreath product $S_n(G)$ and prove that it is a probability measure. The canonical projection
$p_{n,n+1}:\; S_{n+1}(G)\longrightarrow S_n(G)$ is introduced in \textbf{Section \ref{SECTIONCANONICALPROJECTION}}.
In addition, in Section \ref{SECTIONCANONICALPROJECTION} we show that the probability measures
$\mu_{S_n(G)}^{\Ewens}$ are preserved by $p_{n,n+1}$. In \textbf{Section \ref{THESPACEOFVIRTUALPERMUTATIONS}}
we introduce the probability space $\left(\mathfrak{S}_G,\mu^{\Ewens}_{\mathfrak{S}_G}\right)$,
where $\mathfrak{S}_{G}$ is the space of $G$-virtual permutations. We study the action of $S_{\infty}(G)\times S_{\infty}(G)$ on this space and prove that $\mu^{\Ewens}_{\mathfrak{S}_G}$ is quasi-invariant
with respect to this action. The results of Section \ref{THESPACEOFVIRTUALPERMUTATIONS} are used in
\textbf{Section \ref{SECTIONGENERALIZEDREGULARREPRESENTATIONS}} to construct the generalized
regular representations of $S_{\infty}(G)\times S_{\infty}(G)$, and to obtain
some formulae for their characters.

\textbf{Section \ref{SECTIONPROOFT1}} is devoted to
the proof of Theorem \ref{Theorem13.2.2}. To compute the $z$-measures explicitly, we introduce
algebras of symmetric functions $\widehat{\SYM}(\zeta)$, where $\zeta$ belongs to the set $\widehat{G}$
of equivalence classes of continuous unitary irreducible representations o $G$. Theorem \ref{Theoremch1.3.2}
of Section \ref{SECTIONPROOFT1} relates the characters $\chi^{\Lambda}$ of the irreducible
representations of $S_n(G)$ with the products $\prod_{\zeta\in\widehat{G}}\widehat{s}_{\Lambda(\zeta)}(\zeta)$,
where $\widehat{s}_{\lambda}(\zeta)$ is the Schur function in $\widehat{\SYM}(\zeta)$ parameterized by a Young diagram $\lambda$. Then we show how to obtain Theorem \ref{Theorem13.2.2} from Theorem \ref{Theoremch1.3.2}.
Theorem \ref{Theoremch1.3.2} is of independent interest and is proved in \textbf{Section \ref{SECTIONPROOFTHEOREMCH}}.
The proof is based on the Hirai, Hirai, and Hora construction of irreducible
representations of $S_n(G)$.

\textbf{Section \ref{SECTIONHF1}} relates the $z$-measures for $S_n(G)$ with harmonic functions on a certain branching graph, which enables us to obtain an integral representation
for the characters $\chi_z$ of the generalized regular representations in \textbf{Section \ref{SECTIONSPECTRAL}}.
Finally, Theorem \ref{THEOREMSPECTRAL} of Section \ref{SECTIONSPECTRAL} gives a description of the spectral
$z$-measures.

\textbf{Acknowledgement.} I am very grateful to G. Olshanski and A. Sodin for discussions.

\section{Notation}\label{SECTIONNOTATION}
\subsection{The symmetric group and Young diagrams}
For the symmetric group and Young diagrams, we adopt the notation and definitions from Macdonald \cite{Macdonald}.
In particular, let $S_n$ be the symmetric group of degree $n$, that is, the group of permutations of the finite set $I_n=\left\{1,\ldots,n\right\}$. Denote by $\Y_n$ the set of all Young diagrams with  $n$ boxes, and by $\Y$ the set of all Young diagrams. If $\varrho\in\Y_n$, then we can write
$$
\varrho=1^{m_1(\varrho)}2^{m_2(\varrho)}\ldots n^{m_n(\varrho)},
$$
where $m_i(\varrho)$ is the number of rows in $\varrho$ equal to $i$. The number of elements of $S_n$ in the conjugacy class given by $\varrho\in\Y_n$ is
$\frac{n!}{z_{\varrho}}$, where
$$
z_\varrho=1^{m_1(\varrho)}2^{m_2(\varrho)}\ldots n^{m_n(\varrho)}\; m_1(\varrho)! m_2(\varrho)!\ldots m_n(\varrho)!.
$$
The Frobenius formula for the irreducible characters of $S_n$ is
$$
p_{\varrho}=\sum\limits_{\lambda\in\Y_n}\chi^{\lambda}_{\varrho}s_{\lambda},
$$
where $\chi^{\lambda}_{\varrho}$ is the character of the irreducible representation of $S_n$ parameterized  by $\lambda$, and evaluated at the conjugacy class of $S_n$ parameterized by $\varrho$. Finally, if $\lambda\in\Y_n$, then $|\lambda|$ denotes the number of boxes in $\lambda$.
\subsection{The wreath product $S_n\left(G\right)$.}
Let $G$ be a compact group.  Set
\begin{equation}
D_n(G)=G^n=\left\{d=(g_1,\ldots,g_n)\vert g_1\in G,\ldots,g_n\in G\right\}.
\end{equation}
The action of $S_n$ on $D_n(G)=G^n$ is defined by
$$
s:\; d=\left(g_1,\ldots,g_n\right)\longrightarrow s(d)=\left(g_{s^{-1}(1)},\ldots,g_{s^{-1}(n)}\right).
$$
Denote by $S_n(G)$ the wreath product of $G$ with $S_n$. By definition, $S_n(G)$ is the semidirect product $S_n(G)=D_n(G)\rtimes S_n$ (that is, each element of $S_n(G)$ has the form
$(d,s)$, $d\in D_n(G)$, $s\in S_n$, and $(d,s)(d',s')=(ds(d'),ss')$).
In other words,  the wreath product $S_n\left(G\right)$
is a group whose underlying set is
$$
G^n\times S_n=\left\{\left(\left(g_1,\ldots,g_n\right),s\right):\; g_i\in G, s\in S_n\right\}.
$$
The multiplication in $S_n\left(G\right)$ can be expressed explicitly as
$$
\left(\left(g_1,\ldots,g_n\right),s\right)
\left(\left(h_1,\ldots,h_n\right),t\right)
=\left(\left(g_1h_{s^{-1}(1)},\ldots,g_nh_{s^{-1}(n)}\right),st\right).
$$
When $n=1$, $S_1\left(G\right)$ is $G$.
\subsection{Conjugacy classes and irreducible representations of $G$}\label{NotSection1.2}
Denote by $\left[G\right]$ the set of conjugacy classes
of $G$, and by $\widehat{G}$ the set of equivalence classes of continuous irreducible
unitary representations of $G$. We assume that $\widehat{G}$ is at most countable.
Therefore, an element $\zeta\in\widehat{G}$ is an equivalent class of continuous unitary irreducible representations of $G$.  The notation $\pi^{\zeta}\in\zeta$ means that $\pi^{\zeta}$ is a representation of class $\zeta$. We denote by $V^{\zeta}$ the representation space
of $\pi^{\zeta}$, and by $\chi^{\pi^\zeta}$ the character of $\pi^{\zeta}$.
\subsection{Representations of $D_n(G)=G^n$}\label{SubsectionDn}
Denote by $\widehat{D_n(G)}$ the set of equivalence classes of continuous unitary irreducible representations of $D_n(G)$.  We can represent $\widehat{D_n(G)}$ as
$$
\widehat{D_n(G)}=\left\{\eta=\left(\zeta^1,\ldots,\zeta^n\right)\vert
\zeta^1\in\widehat{G},\ldots,\zeta^n\in\widehat{G}\right\}.
$$
Each representation of $D_n(G)$ from the equivalence class $\eta=\left(\zeta^1,\ldots,\zeta^n\right)$ can be written as
\begin{equation}\label{reps1}
\pi^{\zeta^1}\boxtimes\ldots\boxtimes\pi^{\zeta^n},
\end{equation}
where $\pi^{\zeta^j}$ is a representation of $G$ from the equivalence class $\zeta^{j}\in\widehat{G}$, and $\boxtimes$ denotes the outer tensor product of representations
The representation space of (\ref{reps1}) is
\begin{equation}\label{reps3}
V^{\zeta^1}\otimes\ldots\otimes V^{\zeta^n},
\end{equation}
where $\otimes$ denotes the outer tensor product of vector spaces.
The action of the symmetric group $S_n$ on $\widehat{D_n(G)}$  is defined by
$$
s\eta=\left(\zeta^{s^{-1}(1)},\ldots, \zeta^{s^{-1}(n)}\right),\;\;
\eta=\left(\zeta^1,\ldots,\zeta^n\right),\;\; s\in S_n.
$$
We note that the action of $S_n$ on $\widehat{D_n(G)}$ can be understood as
that of $S_n$ on $D_n(G)$:
\begin{equation}\label{reps4}
\left(s\pi^{\eta}\right)(g_1,\ldots,g_n)\overset{\mbox{def}}{=}\pi^{\eta}\left(g_{s(1)},\ldots,g_{s(n)}\right)=\pi^{\zeta^{s^{-1}(1)}}(g_1)
\boxtimes\ldots\boxtimes\pi^{\zeta^{s^{-1}(n)}}(g_n).
\end{equation}
We denote by $S^{\eta}$ a subgroup of $S_n$ defined by
\begin{equation}\label{reps5}
S^{\eta}=\left\{s\in S_n\vert s\eta=\eta\right\}.
\end{equation}
\subsection{The conjugacy classes of $S_n\left(G\right)$}
In order to describe the conjugacy classes of $S_n\left(G\right)$ it is convenient to identify each conjugacy class $c\in\left[G\right]$ with a color. For example, if $G=U(1)$, then $\left[G\right]$
is the unit circle on the complex plane, $\left[G\right]=\left\{c\in\C\vert |c|=1\right\}$, and each point $c$ of the unit circle
is understood as a specific color.

Assume that $x=\left(\left(g_1,\ldots,g_n\right),s\right)\in S_n\left(G\right)$, and write $s\in S_n$ as a product of disjoint cycles. If $\left(i_1i_2\ldots i_r\right)$ is one of these cycles,
its color is determined by that of the conjugacy class of $g_{i_r}g_{i_{r-1}}\ldots g_{i_1}\in G$.
That is, if $g_{i_r}g_{i_{r-1}}\ldots g_{i_1}$ belongs to the conjugacy class $c$, $c\in\left[G\right]$,
then we say that the color of $\left(i_1i_2\ldots i_r\right)$ is $c$.
In this way, the element
$x=\left(\left(g_1,\ldots,g_n\right),s\right)\in S_n\left(G\right)$ gives rise to a Young diagram $\varrho$ with $n$ boxes and with colored rows. Each row of $\varrho$ represents a cycle of $s$, and the color of the row is that of the corresponding cycle.
The Young diagram $\varrho$ can be understood as a collection of colored Young diagrams,
$\varrho=\left(\varrho(c)\right)_{c\in\left[G\right]}$, or as a map
$$
\varrho:\;\; \left[G\right]\longrightarrow \Y,
$$
from the set $\left[G\right]$ of the conjugacy classes of $G$ to the set $\Y$ of all Young diagrams such that
\begin{equation}\label{Condition1}
\sum\limits_{c\in\left[G\right]}\left|\varrho(c)\right|=n.
\end{equation}
The map $\varrho$ is called the \textit{type} of $x$, $x\in S_n\left(G\right)$. It is known that two elements
$x$ and $y$ of $S_n\left(G\right)$ belong to the same conjugacy class of $S_n\left(G\right)$
if and only if they have the same type. The proof of this fact in the case
of an arbitrary compact group can be found in Hora and Hirai \cite{HoraHirai}, Section 1.3.
In what follows, we denote by $\Y_n\left(\left[G\right]\right)$ the set of all
mappings $\varrho$ from $\left[G\right]$ to $\Y$ satisfying condition (\ref{Condition1}).
We can write
\begin{equation}\label{Yn[G]}
\Y_n\left(\left[G\right]\right)
=\left\{\varrho=\left(\varrho(c)\right)_{c\in \left[G\right]}
\biggl|\varrho(c)\in\Y,\;\sum\limits_{c\in\left[G\right]}|\varrho(c)|=n\right\}.
\end{equation}
\begin{example}Let $x=\left(\left(g_1,g_2,g_3,g_4,g_5\right),s\right)$ be an element of $S_5\left(U(1)\right)$, and assume that
$$
g_1=e^{i\pi},\; g_2=e^{i\frac{\pi}{4}},\;g_3= e^{i\frac{3\pi}{4}},\;
g_4=e^{i\pi},\; g_5=e^{i\frac{2\pi}{3}},
$$
and that $s=(152)(34)$. Then the color of $(34)$ is $g_4g_3=e^{i\frac{7\pi}{4}}$,
and the color of $(152)$ is $g_2g_5g_1=e^{i\frac{23\pi}{12}}$. We obtain
the Young diagram $\varrho=(3,2)$, where the first row is of color $c=e^{i\frac{23\pi}{12}}$, and the second row is of color $c=e^{i\frac{\pi}{2}}$.
The Young diagram $\varrho$ can be understood as a map from the unit circle to $\Y$ such that
$$
\varrho(c)=
\left\{
  \begin{array}{ll}
    \Box\Box\Box, & c=e^{i\frac{23\pi}{12}},\\
    \Box\Box, & c=e^{i\frac{7\pi}{4}}, \\
    \emptyset, & \hbox{otherwise.}
  \end{array}
\right.
$$
The map $\varrho$ is the type of $x=\left(\left(g_1,g_2,g_3,g_4,g_5\right),s\right)$, and determines the conjugacy class of
$x=\left(\left(g_1,g_2,g_3,g_4,g_5\right),s\right)$ in $S_5\left(U(1)\right)$.
\end{example}
\subsection{Irreducible unitary representations of $S_n\left(G\right)$}
It is known (see Hirai, Hirai, Hora \cite{HiraiHiraiHoraI}, $\S 3$ and $\S 4$, and Hora and Hirai \cite{HoraHirai}, Section 1.4) that
the equivalence classes of irreducible unitary representations of $S_n\left(G\right)$ are parameterized by maps
$$
\Lambda:\; \widehat{G}\longrightarrow\Y
$$
such that
$$
\sum\limits_{\zeta\in \widehat{G}}\left|\Lambda(\zeta)\right|=n.
$$
The set of all such maps is denoted by $\Y_n\left(\widehat{G}\right)$. This set can be represented as
\begin{equation}\label{YnG}
\Y_n\left(\widehat{G}\right)=\left\{\Lambda=\left(\Lambda(\zeta)\right)_{\zeta\in\widehat{G}}
\biggl|\Lambda(\zeta)\in\Y,\;\sum\limits_{\zeta\in\widehat{G}}|\Lambda(\zeta)|=n\right\}.
\end{equation}
\section{The Ewens measure on $S_n\left(G\right)$}\label{SECTIONEWENSMEASURE}
Equip the group $S_n\left(G\right)$ with a measure $\mu_{S_n\left(G\right)}$ which is the product measure
of the normalized Haar measures on $G$, and the uniform measure on $S_n$. In particular, if
$f:\; S_n\left(G\right)\longrightarrow\C$ is  a continuous function defined on $S_n\left(G\right)$, then we have
\begin{equation}
\int\limits_{S_n\left(G\right)}f(x)d\mu_{S_n\left(G\right)}(x)=
\frac{1}{n!}\sum\limits_{s\in S_n}\int\limits_{G}\ldots\int\limits_{G}
f\left(\left(\left(g_1,\ldots,g_n\right),s\right)\right)
d\mu_G\left(g_1\right)\ldots d\mu_G\left(g_n\right),
\end{equation}
where $\mu_G$ is the normalized Haar measure on $G$. We agree that if $x=\left(\left(g_1,\ldots,g_n\right),s\right)\in
S_n\left(G\right)$, then $[x](c)$ denotes the number of cycles of color $c$ in $s$, $s\in S_n$.
\begin{defn}\label{Definition2.1}Let $z:\; G\longrightarrow\C\setminus\{0\}$ be a continuous central function.
Set
\begin{equation}
P_{S_n\left(G\right)}^{\Ewens}(x)=n!\;\frac{\prod\limits_{c\in[G]}\left(z(c)\right)^{[x](c)}\left(\overline{z(c)}\right)^{[x](c)}}{
I(I+1)(I+n-1)
},\; x\in S_n(G),
\end{equation}
where $I=\int\limits_{G}\left|z(g)\right|^2d\mu_{G}(g)$. The measure
$\mu^{\Ewens}_{S_n(G)}$ defined by
\begin{equation}\label{TheRelationBetweenMeasures}
d\mu_{S_n\left(G\right)}^{\Ewens}(x)=P_{S_n\left(G\right)}^{\Ewens}(x)d\mu_{S_n\left(G\right)}(x)
\end{equation}
is called the Ewens measure on $S_n\left(G\right)$.
\end{defn}
\begin{prop}The Ewens measure $\mu_{S_n\left(G\right)}^{\Ewens}$ is a probability measure
on $S_n\left(G\right)$.
\end{prop}
\begin{proof}Set $t(c)=z(c)\overline{z(c)}$, and write $x=\left(\left(g_1,\ldots,g_n\right),s\right)$. Assume that
$s\in S_n$ can be represented as
$$
s=\left(i_1i_2\ldots i_l\right)\left(j_1j_2\ldots j_m\right)\ldots\left(\nu_1\nu_2\ldots\nu_r\right).
$$
Then
$$
\prod\limits_{c\in\left[G\right]}\left(t(c)\right)^{[x](c)}=
t\left(g_{i_l}\ldots g_{i_2}g_{i_1}\right)\ldots
t\left(g_{j_m}\ldots g_{j_2}g_{j_1}\right)t\left(g_{\nu_r}\ldots g_{\nu_2}g_{\nu_1}\right).
$$
This gives
\begin{equation}
\int\limits_{G}\ldots\int\limits_{G}\prod\limits_{c\in\left[G\right]}
\left(t(c)\right)^{[x](c)}d\mu_{G}(g_1)\ldots d\mu_G(g_n)= I^{m_1+m_2+\ldots+m_n},
\end{equation}
where $I$ is defined as in Definition \ref{Definition2.1},
and $m_i$ denotes the number of cycles of length $i$ in $s$. As a result, we obtain
\begin{equation}\label{TRIPYAT}
\begin{split}
\int\limits_{S_n\left(G\right)}d\mu^{\Ewens}_{S_n\left(G\right)}(x)&=
\frac{1}{I(I+1)\ldots (I+n-1)}\\
&\times\sum\limits_{m_1+2m_2+\ldots+nm_n=n}
\frac{n!I^{m_1+m_2+\ldots+m_n}}{1^{m_1}m_1!2^{m_2}m_2!\ldots n^{m_n}m_n!},
\end{split}
\end{equation}
where $\frac{n!}{1^{m_1}m_1!2^{m_2}m_2!\ldots n^{m_n}m_n!}$ is the number of elements in $S_n$ whose cycle structure is $1^{m_1}2^{m_2}\ldots n^{m_n}$.
The expression on the right-hand side of equation (\ref{TRIPYAT}) can be rewritten as
$$
\sum\limits_{s\in S_n}P_{n,I}^{\Ewens}(s)=1,
$$
where $P^{\Ewens}_{n,I}$ is defined by (\ref{Res.1}) (with $\theta$ replaced by $I$).
\end{proof}
\begin{rem}
(a) Assume that $G$ contains only the unit element. Then $z$ is a non-zero complex parameter, $S_n(G)$ can be identified with the symmetric group $S_n$, and $\mu_{S_n(G)}^{\Ewens}$ turns into the Ewens measure on the symmetric group defined by equation (\ref{Res.1}) with the parameter $\theta=z\bar{z}$.\\
(b) Let $G$ be a finite group, and let $c_1$, $\ldots$, $c_k$ be the conjugacy classes of $G$. Denote by $z_l$ the value of the function $z$ on the conjugacy class $c_l$.
Then
$$
I=\sum\limits_{l=1}^k\frac{t_l}{\zeta_{c_l}},
$$
where $\zeta_{c_l}=\frac{|G|}{|c_l|}$ and $t_l=z_l\bar{z}_l$.  We obtain
$$
\mu_{S_n(G)}^{\Ewens}(x)=\frac{t_1^{[x](c_1)}\ldots t_k^{[x](c_k)}}{|G|^n
\left(\frac{t_1}{\zeta_{c_1}}+\ldots+\frac{t_k}{\zeta_{c_k}}\right)_n},
$$
where $(b)_n=b(b+1)\ldots(b+n-1)$ is the Pochhammer symbol.
This is the Ewens distribution on the wreath product of a finite group $G$ with the symmetric group $S_n$ introduced in Strahov \cite{StrahovMPS}, Proposition 4.1.
\end{rem}
\section{The canonical projection}\label{SECTIONCANONICALPROJECTION}
Here we define the canonical projection
$$
p_{n,n+1}:\; S_{n+1}\left(G\right)\longrightarrow S_n\left(G\right).
$$
The definition below is the same as that for the wreath products of a finite group with symmetric groups; see Strahov \cite{StrahovIsr25}, Section 2.2.

Let $\tilde{x}=\left(\left(g_1,\ldots,g_{n+1}\right),\tilde{s}\right)$ be an element of $S_{n+1}(G)$.
Represent $\tilde{s}$ in terms of cycles. If $n+1$ is a fixed point of $\tilde{s}$, then we set
$s=\tilde{s}$, and $p_{n,n+1}\left(\tilde{x}\right)=\left(\left(g_1,\ldots,g_{n}\right), s\right)$.
If $n+1$ belongs to a cycle,
\begin{equation}\label{3.1}
i_1\rightarrow\ldots\rightarrow i_m\rightarrow n+1\rightarrow i_{m+1}\rightarrow\ldots\rightarrow i_r,
\end{equation}
then we remove $n+1$ out of the cycle, and replace
\begin{equation}\label{3.2}
\tilde{g}=\left(g_1,\ldots,g_n,g_{n+1}\right)
\end{equation}
by
\begin{equation}\label{3.3}
g=\left(g_1,\ldots, g_{i_{m+1}}g_{n+1},\ldots,g_n\right).
\end{equation}
Thus, $g$ is obtained from $\tilde{g}$ by removing the $n+1$th element $g_{n+1}$ from $\tilde{g}$, and by
replacing the element $g_{i_{m+1}}$ of $\tilde{g}$ by $g_{i_{m+1}}g_{n+1}$.

An important property of the canonical projection is that it \textit{preserves the colors of the cycles}. In fact, the color of the cycle (\ref{3.1}) is determined by the conjugacy class of $g_{i_r}\ldots g_{i_{m+1}}g_{n+1}g_{i_m}\ldots g_{i_1}$
in $G$. Under the canonical projection the cycle (\ref{3.1}) turns into the cycle
\begin{equation}\label{3.4}
i_1\rightarrow\ldots\rightarrow i_m\rightarrow i_{m+1}\rightarrow\ldots\rightarrow i_r,
\end{equation}
the element $\tilde{g}$ defined by (\ref{3.2}) turns into the element $g$ defined by (\ref{3.3}), so the color of the cycle (\ref{3.4}) is again determined by the conjugacy class of $g_{i_r}\ldots g_{i_{m+1}}g_{n+1}g_{i_m}\ldots g_{i_1}$
in $G$.
\begin{prop}\label{PROPOSITIONCONISTECYEWENS}We have
\begin{equation}
\int\limits_{O(x_n)}d\mu^{\Ewens}_{S_{n+1}(G)}(x_{n+1}')=d\mu_{S_n(G)}^{\Ewens}(x_n),
\end{equation}
where  $O(x_n)$ is a subset of $S_{n+1}(G)$ defined by
\begin{equation}
O(x_n)=\left\{x_{n+1}':\; x_{n+1}'\in S_{n+1}(G),\; p_{n,n+1}(x_{n+1}')=x_n\right\}.
\end{equation}
In other words, the canonical projections $p_{n-1,n}$ preserve the measures $\mu^{\Ewens}_{S_{n}(G)}$.
\end{prop}
\begin{proof}Fix $x_n=\left(\left(g_1,\ldots,g_n\right),s\right)\in S_n\left(G\right)$. Assume that
$x_{n+1}'=\left(\left(g_1',\ldots,g_{n+1}'\right),s'\right)\in S_{n+1}(G)$ is such that $p_{n,n+1}(x_{n+1}')=x_n$, and
such that $n+1$ is a fixed point of $s'$. Then $x_{n+1}'$ can be written as
$$
x_{n+1}'=\left(\left(g_1,\ldots,g_n,g_{n+1}'\right),s(n+1)\right).
$$
Denote by $O'(x_n)$ the subset of $O(x_n)$ consisting of all such elements. If $x_{n+1}'\in O'(x_n)$, then
\begin{equation}\label{10.1.5}
\left[x_{n+1}'\right](c)=\left\{
\begin{array}{ll}
\left[x_{n}\right](c)+1, &  g_{n+1}'\in c,\\
\left[x_{n}\right](c), & \mbox{otherwise}.
\end{array}
\right.
\end{equation}
This gives
\begin{equation}\label{10.1.6}
\prod\limits_{c\in\left[G\right]}\left(z(c)\right)^{\left[x_{n+1}'\right](c)}
\left(\overline{z(c)}\right)^{\left[x_{n+1}'\right](c)}
=z\left(g_{n+1}'\right)\overline{z\left(g_{n+1}'\right)}
\prod\limits_{c\in\left[G\right]}\left(z(c)\right)^{\left[x_{n}\right](c)}
\left(\overline{z(c)}\right)^{\left[x_{n}\right](c)},
\end{equation}
where we have used the fact that  $z:\; G\rightarrow \C\setminus\{0\}$ is a central function.
Taking into account (\ref{10.1.6}) we find
\begin{equation}\label{O(1)}
\begin{split}
&\int\limits_{O'(x_n)}d\mu^{\Ewens}_{S_{n+1}(G)}(x_{n+1}')=\frac{\int\limits_{G}z\left(g_{n+1}'\right)\overline{z\left(g_{n+1}'\right)}d\mu_{G}\left(g_{n+1}'\right)}
{I+n}d\mu_{S_n(G)}^{\Ewens}(x_n)\\
&=\frac{I}{I+n}d\mu_{S_n(G)}^{\Ewens}\left(x_n\right),
\end{split}
\end{equation}
where we have used the explicit form of the Ewens measures $\mu_{S_n(G)}^{\Ewens}$, see Definition \ref{Definition2.1}.

Assume that
$x_{n+1}'=\left(\left(g_1',\ldots,g_{n+1}'\right),s'\right)\in S_{n+1}(G)$ is such that $p_{n,n+1}(x_{n+1}')=x_n$, and such that $s'$ has the form
\begin{equation}
s'=\left(i_1i_2\ldots i_r\right)\ldots\left(\nu_1\nu_2\ldots\nu_mn+1\nu_{m+1}\ldots\nu_l\right)\ldots\left(j_1j_2\ldots j_p\right),
\end{equation}
i.e. $n+1$ is situated inside of an existing cycle of $s$. Then $s$ is given by
\begin{equation}
s=\left(i_1i_2\ldots i_r\right)\ldots\left(\nu_1\nu_2\ldots\nu_m\nu_{m+1}\ldots\nu_l\right)\ldots\left(j_1j_2\ldots j_p\right),
\end{equation}
and
\begin{equation}
g_1=g_1',\ldots,g_{\nu_{m+1}}=g'_{\nu_{m+1}}g_{n+1}',\ldots,g_n=g_n'.
\end{equation}
Denote by $O''(x_n)$ the set of all such elements $x_{n+1}'$. Since the canonical projection preserves the colors of
cycles, and since for the elements of $O''(x_n)$ the number of cycles in $s'$ is equal to the number of cycles in $s$, we find
\begin{equation}\label{O(2)}
\begin{split}
&\int\limits_{O''(x_n)}d\mu^{\Ewens}_{S_{n+1}(G)}(x_{n+1}')=\frac{n\int\limits_{G}d\mu_{G}\left(g_{n+1}'\right)}
{I+n}d\mu_{S_n(G)}^{\Ewens}(x_n)\\
&=\frac{n}{I+n}d\mu_{S_n(G)}^{\Ewens}\left(x_n\right).
\end{split}
\end{equation}
Here we have used the fact that there are $n$ possibilities to insert $n+1$ into an existing cycle of $s$, $s\in S_{n}$. Finally,
\begin{equation}
\begin{split}
&\int\limits_{O(x_n)}d\mu^{\Ewens}_{S_{n+1}(G)}(x_{n+1}')=
\int\limits_{O'(x_n)}d\mu^{\Ewens}_{S_{n+1}(G)}(x'_{n+1})+\int\limits_{O''(x_n)}d\mu^{\Ewens}_{S_{n+1}(G)}(x_{n+1}')\\
&=d\mu_{S_n(G)}^{\Ewens}\left(x_n\right),
\end{split}
\end{equation}
as it follows from equations (\ref{O(1)}) and (\ref{O(2)}).
\end{proof}
\begin{prop} The projection $p_{n,n+1}:\; S_{n+1}\left(G\right)\longrightarrow S_n\left(G\right)$ is equivariant with respect to the two-sided action of $S_n\left(G\right)$, i.e.
$$
p_{n,n+1}\left((\kappa,\pi)(g,s)(h,t)\right)=(\kappa,\pi)p_{n,n+1}\left((g,s)\right)(h,t),
$$
for each $(g,s)\in G_{n+1}(G)$, and each $(\kappa,\pi)\in S_n(G)$, $(h,t)\in S_n(G)$.
\end{prop}
\begin{proof}
The proof is the same as in the case of wreath products of a finite group $G$ with the symmetric groups
$S_n$, see Proposition 2.1 in Strahov \cite{StrahovIsr25}.
\end{proof}
\section{The space of $G$-virtual permutations}\label{THESPACEOFVIRTUALPERMUTATIONS}
\subsection{The probability space $\left(\mathfrak{S}_{G},\mu_{\mathfrak{S}_G}^{\Ewens}\right)$}\label{Section4.1}
The space $\mathfrak{S}_{G}$ of $G$-virtual permutations is introduced in the same way as in the case of a finite group $G$, see Strahov \cite{StrahovIsr25}, Section 2.3. Namely, we consider the sequence of canonical projections
$$
S_1(G)\overset{p_{1,2}}{\longleftarrow}\ldots\overset{p_{n-1,n}}{\longleftarrow} S_n(G)\overset{p_{n,n+1}}{\longleftarrow} S_{n+1}(G)\overset{p_{n+1,n+2}}{\longleftarrow}\ldots,
$$
and set
$$
\mathfrak{S}_G=\underset{\longleftarrow}{\lim}\;S_n(G).
$$
Thus $\mathfrak{S}_G$ is the projective limit of the sets $S_n(G)$.  By definition, the elements of $\mathfrak{S}_{G}$
are sequences $x=\left(x_1,x_2,\ldots\right)$ such that $x_n\in S_n\left(G\right)$, and $p_{n,n+1}\left(x_{n+1}\right)=x_n$.

By Proposition \ref{PROPOSITIONCONISTECYEWENS} the family $\left(\mu_{S_n(G)}^{\Ewens}\right)_{n=1}^{\infty}$
is consistent with the canonical projection $p_{n,n+1}:\; S_{n+1}(G)\longrightarrow S_n(G)$.
Therefore, the sequence $\left(\mu^{\Ewens}_{S_n(G)}\right)_{n=1}^{\infty}$ gives rise to a probability measure
$$
\mu_{\mathfrak{S}_G}^{\Ewens}=\underset{\longleftarrow}{\lim}\;\mu^{\Ewens}_{S_n(G)}
$$
on the space of virtual $G$-permutations $\mathfrak{S}_{G}$.
\subsection{The wreath product $S_{\infty}(G)$. The action of $S_{\infty}(G)\times S_{\infty}(G)$
on $\mathfrak{S}_G$
}\label{SubSectionWreathInfty}
Let $S(\infty)$ be the group of finite permutations of the set $\left\{1,2,\ldots\right\}$, and $G$ be a compact group.
Denote by $D_{\infty}\left(G\right)$ the restricted direct product of $G$, i.e.
$$
D_{\infty}\left(G\right)=\left\{g=\left(g_1,g_2,\ldots\right)\in G^{\infty}:\;g_j=e_{G}\;\mbox{except  finitely many j's}
\right\}.
$$
Here $e_{G}$ denotes the unit element of $G$. The infinite symmetric group $S(\infty)$ acts on
$D_{\infty}\left(G\right)$ according to the formula
\begin{equation}\label{3Action}
s(g)=\left(g_{s^{-1}(1)},g_{s^{-1}(2)},\ldots \right),\; s\in S(\infty),\; g\in D_{\infty}\left(G\right).
\end{equation}
\begin{defn}
The wreath product $S_{\infty}(G)$ of a compact group $G$ with the infinite symmetric group $S(\infty)$ is the semidirect product of $D_{\infty}\left(G\right)$ with $S(\infty)$ defined by action (\ref{3Action}). The underlying set of $S_{\infty}(G)$ is $D_{\infty}\left(G\right)\times S(\infty)$, with multiplication defined by
$$
\left(g,s\right)\left(h,t\right)=\left(gs(h),st\right),
$$
where $g, h\in D_{\infty}\left(G\right)$, and $s, t\in S(\infty)$.
\end{defn}
Under the canonical inclusion $i_n:\; S_{n}(G)\longrightarrow S_{\infty}(G)$, we can regard $S_{\infty}(G)$ as $\bigcup_{n=1}^{\infty}S_n(G)$. In addition, the group $S_{\infty}(G)$ can be identified with the subset of $\mathfrak{S}_{G}$ consisting of stable sequences $\left(x_n\right)$ such that
$$
x_n=\left(g,s\right), \;\; \left(g,s\right)\in D_{\infty}\left(G\right)\times S(\infty),
$$
for sufficiently large $n$.

Let $W=\left(w_1,w_2\right)\in S_{\infty}(G)\times S_{\infty}(G)$. The right action of $S_{\infty}(G)\times S_{\infty}(G)$ on $\mathfrak{S}_{G}$ is defined as
$$
xW=y,\; x=\left(x_1,x_2,\ldots\right),\;y=\left(y_1,y_2,\ldots\right),
$$
where $y_n=w_2^{-1}x_nw_1$ for all large enough $n$. Specifically, the equality just written above holds whenever $n$ is so large that both $w_1$, $w_2$ are already in $S_n(G)$.
\subsection{The fundamental cocycles of the dynamical system $\left(\mathfrak{S}_G,S_{\infty}(G)\times S_{\infty}(G)\right)$.}
For each $c\in[G]$ define a function $C_c(x,W)$ on $\mathfrak{S}_G\times\left(S_{\infty}(G)\times S_{\infty}(G)\right)$ by
\begin{equation}\label{11.3.1}
C_c(x,W)=\left[p_n(xW)\right](c)-\left[p_n(x)\right](c).
\end{equation}
Here $p_n:\;\mathfrak{S}_G\longrightarrow S_n(G)$ is the natural projection, and the
action of $W\in S_{\infty}(G)\times S_{\infty}(G)$ on $x\in\mathfrak{S}_G$ is defined in Section
\ref{SubSectionWreathInfty}. In equation (\ref{11.3.1}) $n$ is large enough so that $W\in S_n(G)\times S_n(G)$.
\begin{prop}\label{Prop4.2} The quantities $C_c(x,W)$ defined by equation (\ref{11.3.1}) do not depend on $n$ provided that $n$ is so large that the element $W$ of $S_{\infty}(G)\times S_{\infty}(G)$
already belongs to $S_n(G)\times S_n(G)$.
\end{prop}
\begin{proof} The proof of this Proposition is a word-by-word repetition of the arguments in the case of a finite $G$, see the proof of Theorem 4.1 in Strahov \cite{StrahovIsr25}.
\end{proof}
\begin{prop}\label{Prop12.1.1}
For each continuous function $z:\; [G]\longrightarrow\C\setminus\{0\}$,
the measure $\mu_{\mathfrak{S}_G}^{\Ewens}$ is quasi-invariant with respect to the action of $S_{\infty}(G)\times S_{\infty}(G)$ on $\mathfrak{S}_G$. More precisely,
the Radon-Nikod$\acute{y}$m derivative is given by
\begin{equation}
\frac{d\mu^{\Ewens}_{\mathfrak{S}_G}(xW)}{d\mu_{\mathfrak{S}_G}^{\Ewens}(x)}=\prod\limits_{c\in [G]}\left(z(c)\overline{z(c)}\right)^{C_c(x,W)},
\end{equation}
where $x\in\mathfrak{S}_G$, $W\in S_{\infty}(G)\times S_{\infty}(G)$, and $C_c(x,W)$ is defined by equation  (\ref{11.3.1}).
\end{prop}
\begin{proof} The statement of Proposition \ref{Prop12.1.1} is equivalent to the fact that the equation
\begin{equation}\label{12.1.3}
d\mu_{\mathfrak{S}_G}^{\Ewens}\left(VW\right)=\int\limits_{V}\prod\limits_{c\in [G]}\left(z(c)\overline{z(c)}\right)^{C_{c}(x,W)}d\mu_{\mathfrak{S}_G}^{\Ewens}(x)
\end{equation}
holds true for each Borel set $V$ of $\mathfrak{S}_G$. The Borel sets of $\mathfrak{S}_G$ are generated by cylinder sets, and each cylinder set is a disjoint union of the sets of the form
\begin{equation}\label{12.1.4}
V_n(y)=\left\{\left(x_1,x_2,\ldots\right)\in\mathfrak{S}_G,\; x_n=y\right\},\; y\in S_n(G).
\end{equation}
Thus, fix $W\in S_{\infty}(G)\times S_{\infty}(G)$, and choose $m$ so large that $W\in S_m(G)\times S_m(G)$. Our aim is to show that for each $n\geq m$ and for each $y\in S_n(G)$ condition (\ref{12.1.3}) is true for $V=V_n(y)$, where $V_n(y)$ is defined by
(\ref{12.1.4}). From Proposition \ref{Prop4.2} we conclude that
\begin{equation}\label{12.1.5}
\prod\limits_{c\in[G]}\left(z(c)\overline{z(c)}\right)^{C_c(x,W)}
=\prod\limits_{c\in[G]}\left(z(c)\overline{z(c)}\right)^{\left[yW\right](c)-\left[y\right](c)},
\end{equation}
for each $x\in V_n(y)$. This gives
\begin{equation}\label{12.1.6}
\begin{split}
&\int_{V_n(y)}\prod\limits_{c\in[G]}\left(z(c)\overline{z(c)}\right)^{C_c(x,W)}
d\mu^{\Ewens}_{\mathfrak{S}_G}(x)\\
&=\prod\limits_{c\in[G]}\left(z(c)\overline{z(c)}\right)^{\left[yW\right](c)-[y](c)}d\mu^{\Ewens}_{\mathfrak{S}_G}(V_n(y)).
\end{split}
\end{equation}
By the very definition of $\mu_{\mathfrak{S}_G}^{\Ewens}$ in Section \ref{Section4.1} we find
\begin{equation}
d\mu_{\mathfrak{S}_G}^{\Ewens}\left(V_n(y)\right)
=d\mu_{\mathfrak{S}_G}^{\Ewens}(y).
\end{equation}
Taking into account the definition of $\mu_{\mathfrak{S}_G}^{\Ewens}$ (see Definition \ref{Definition2.1}), we find that the right-hand side of (\ref{12.1.6})
is equal to $d\mu_{\mathfrak{S}_G}^{\Ewens}(yW )$. On the other hand,
$$
d\mu_{\mathfrak{S}_G}^{\Ewens}\left(V_n(y)W\right)=d\mu_{\mathfrak{S}_G}^{\Ewens}\left(V_n(yW)\right)=d\mu_{S_n(G)}^{\Ewens}(yW),
$$
so equation (\ref{12.1.3}) holds true  for $V=V_n(y)$.
\end{proof}
\section{Generalized regular representations and their characters}\label{SECTIONGENERALIZEDREGULARREPRESENTATIONS}
In this Section we present a construction of generalized regular representations of $S_{\infty}(G)\times S_{\infty}(G)$ in the case of a compact group $G$. The construction is a generalization of that in Strahov
\cite{StrahovIsr25}, \S 6,  where the case of a finite group $G$ was considered. Ref. \cite{StrahovIsr25} contains a review of the basic notation and results used in harmonic analysis on big groups, and, in particular,
in Sections \ref{12.2.1} and \ref{Section12.4} below.
\subsection{Definition of $\left(T_z, L^2\left(\mathfrak{S}_G,\mu^{\Ewens}_{\mathfrak{S}_G}\right)\right)$}
\begin{defn}\label{12.2.1}
Let $z:\;[G]\longrightarrow\C\setminus\{0\}$ be a continuous function and let $
\mu^{\Ewens}_{\mathfrak{S}_G}$ be the Ewens measure on $\mathfrak{S}_G$ associated with this function. Set
\begin{equation}\label{12.2.2}
\left(T_z\left(W\right)f\right)(x)=f(xW)\prod\limits_{c\in[G]}\left(z(c)\right)^{C_c(x,W)},
\end{equation}
where $W\in S_{\infty}(G)\times S_{\infty}(G)$, $f\in L^2\left(\mathfrak{S}_G,\mu_{\mathfrak{S}_G}^{\Ewens}\right)$, the function $C_c(x,W)$ is defined by equation (\ref{11.3.1}),
and the action of $S_{\infty}(G)\times S_{\infty}(G)$ on $\mathfrak{S}_G$ is defined in Section \ref{SubSectionWreathInfty}. Since
\begin{equation}\label{12.3.1}
C_c\left(x,W_2W_1\right)=C_c\left(xW_2,W_1\right)+C_c\left(x,W_2\right)
\end{equation}
is satisfied for all $W_1, W_2\in S_{\infty}(G)\times S_{\infty}(G)$, and all $x\in\mathfrak{S}_G$,  equation (\ref{12.2.2}) correctly defines a representation
$\left(T_z, L^2\left(\mathfrak{S}_G,\mu^{\Ewens}_{\mathfrak{S}_G}\right)\right)$
of $S_{\infty}(G)\times S_{\infty}(G)$. This representation is called the \textit{generalized regular representation} of $S_{\infty}(G)\times S_{\infty}(G)$.
\end{defn}
Note that Proposition \ref{Prop12.1.1} implies that this representation is unitary. Moreover, if $z(c)=1$ for each $c\in [G]$, then this representation becomes the biregular representation of $S_{\infty}(G)\times S_{\infty}(G)$.
It is important that $\left(T_z, L^2\left(\mathfrak{S}_G,\mu^{\Ewens}_{\mathfrak{S}_G}\right)\right)$ can be understood as the inductive limit of the representations
\begin{equation}\label{InductiveSequence}
T_z\biggr|_{L^2\left(S_1(G),\mu^{\Ewens}_{S_1(G)}\right)},\;\;
T_z\biggr|_{L^2\left(S_2(G),\mu^{\Ewens}_{S_2(G)}\right)},\ldots,
\end{equation}
see Section 6 in Strahov \cite{StrahovIsr25}.
\subsection{The character $\chi_{z}$ of $T_z$}\label{Section12.4}
Let $\zeta_0\in L^2\left(\mathfrak{S}_G,\mu^{\Ewens}_{\mathfrak{S}_G}\right)$
be the function identically equal  to $1$ on $\mathfrak{S}_G$. This function is invariant under the action of
$$
\diag\left(S_{\infty}(G)\right)=\left\{W=(w,w):\;w\in S_{\infty}(G)\right\},
$$
and has the norm $1$. Therefore, $\left(T_z,\zeta_0\right)$ is a spherical representation of the Gelfand pair $\left(S_{\infty}(G)\times S_{\infty}(G),\diag\left(S_{\infty}\left(G\right)\right)\right)$. The corresponding spherical functions is defined by
\begin{equation}\label{12.4.1}
\Phi_z(W)=\left<T_z(W)\zeta_0,\zeta_0\right>_{L^2\left(\mathfrak{S}_G,\mu_{\mathfrak{S}_G}^{\Ewens}\right)},\;\; W\in S_{\infty}(G)\times S_{\infty}(G).
\end{equation}
The \textit{character} $\chi_z$ of $T_z$ is defined by
\begin{equation}\label{12.4.2}
\chi_z(w)=\Phi_z\left(w,e_{S_{\infty}(G)}\right)=
\left<T_z\left(w,e_{S_{\infty}(G)}\right)\zeta_0,\zeta_0\right>_{L^2\left(\mathfrak{S}_G,\mu_{\mathfrak{S}_G}^{\Ewens}\right)},
\end{equation}
where $w\in S_{\infty}(G)$, and $e_{S_{\infty}(G)}$ is the unit element
of $S_{\infty}(G)$.

Our aim is to describe the expansion of $\chi_z|_{S_n(G)}$  into irreducible characters of $S_n(G)$. First, we note that $\zeta_0$ can be understood as an element of $L^2\left(S_n(G),\mu^{\Ewens}_{S_n(G)}\right)$
for each $n=1,2,\ldots$. Since $\mu_{\mathfrak{S}_G}^{\Ewens}=
\underset{\longleftarrow}{\lim}\mu_{S_n(G)}^{\Ewens}$, and since
$T_z$ is the inductive limit of representations  (\ref{InductiveSequence})
we can write
\begin{equation}\label{12.5.1}
\chi_z|_{S_n(G)}(x)=\left<T_z\left(x,e_{S_n(G)}\right)\zeta_0,\zeta_0\right>_{
L^2\left(S_n(G),\mu^{\Ewens}_{S_n(G)}\right)},\; x\in S_n(G).
\end{equation}
Second, introduce the function $F_z^{(n)}:\; S_n(G)\longrightarrow\C$ by
\begin{equation}\label{12.6.1}
F_z^{(n)}(x)=\frac{\left(n!\right)^{\frac{1}{2}}}{\left(\left(I\right)_n\right)^{\frac{1}{2}}}\;\prod\limits_{c\in [G]}
\left(z(c)\right)^{[x](c)},
\end{equation}
where $I=\int\limits_{G}|z(g)|^2d\mu_G(g)$, and $(I)_n=I(I+1)\ldots (I+n-1)$. In addition, let us introduce
the biregular representation $\Reg_n$ of $S_n(G)\times S_n(G)$. This representation acts in the space
$L^2\left(S_n(G),\mu_{S_n(G)}^{\Ewens}\right)$, and it is defined by
\begin{equation}\label{12.6.2}
\left(\Reg_n(Y)f\right)(x)=f(xY),\;\;Y=(y_1,y_2)\in S_n(G)\times S_n(G),\;\; xY=y_2^{-1}xy_1,
\end{equation}
where $x\in S_n(G)$.
\begin{prop}\label{Prop12.6.3}
We have
\begin{equation}\label{12.6.4}
\left(\Reg_n(Y)F_z^{(n)}\varphi\right)(x)=\left(F_z^{(n)}T_z(Y)\varphi\right)(x),
\end{equation}
for each $\varphi\in L^2\left(S_n(G),\mu_{S_n(G)}^{\Ewens}\right)$, each $Y\in S_n(G)\times S_n(G)$, and each $x\in S_n(G)$. Here $F_z^{(n)}$ denotes the operation of multiplication by $F_z^{(n)}(x)$.
\end{prop}
\begin{proof} We represent the right-hand side of (\ref{12.6.4}) as
$$
F_z^{(n)}(x)\varphi\left(xY\right)\frac{F_z^{(n)}(xY)}{F_z^{(n)}(x)}=
\varphi\left(y_2^{-1}xy_1\right)F_z^{(n)}\left(y_2^{-1}xy_1\right).
$$
On the other hand,
$$
\left(\Reg_n(Y)F_z^{(n)}\varphi\right)(x)=F_z^{(n)}\left(y_2^{-1}xy_1\right)\varphi\left(y_2^{-1}xy_1\right),
$$
by the very definition of $\Reg_n$, see equation (\ref{12.6.2}).
\end{proof}
Using Proposition \ref{Prop12.6.3} we obtain
\begin{equation}\label{12.6.5}
\begin{split}
&\chi_z\vert_{S_n(G)}(w)\\
&=\left<T_z\left(w,e_{S_n(G)}\right)\zeta_0,\zeta_0\right>_{L^2\left(S_n(G),\mu_{S_n(G)}^{\Ewens}\right)}
\\
&=\left<F_z^{(n)}T_z\left(w,e_{S_n(G)}\right)\zeta_0,F_z^{(n)}\zeta_0\right>_{L^2\left(S_n(G),\mu_{S_n(G)}\right)}
\\
&=\left<\Reg_n\left((w,e_{S_n(G)}\right)F_z^{(n)}\zeta_0,F_z^{(n)}\zeta_0\right>_{L^2\left(S_n(G),\mu_{S_n(G)}\right)}\\
&=\int\limits_{S_n(G)}F_z^{(n)}(xw)\overline{F_z^{(n)}(x)}d\mu_{S_n(G)}(x),
\end{split}
\end{equation}
were we have used the relation between $\mu^{\Ewens}_{S_n(G)}$ and $\mu_{S_n(G)}$, equation (\ref{TheRelationBetweenMeasures}).
We then conclude that $\chi_z\vert_{S_n(G)}$ has the following representation
\begin{equation}\label{12.6.6}
\chi_z\vert_{S_n(G)}(w)=\frac{n!}{(I)_n}\int\limits_{S_n(G)}\prod\limits_{c\in [G]}
\left(z(c)\right)^{[xw](c)}\left(\overline{z(c)}\right)^{[x](c)}d\mu_{S_n(G)}(x).
\end{equation}
Given $\Lambda\in\Y_n\left(\widehat{G}\right)$ denote by $\chi^{\Lambda}$ the character of an irreducible unitary representation of $S_n(G)$ whose equivalence class is parameterized by $\Lambda$. The function
\begin{equation}\label{ch1.1.1}
\varphi_z(x)=\prod\limits_{c\in [G]}\left(z(c)\right)^{[x](c)},\;\; x\in S_n(G).
\end{equation}
can be understood as a central function on $S_n(G)$, and it can be expanded as
\begin{equation}\label{12.6.7}
\varphi_z(x)
=\sum\limits_{\Lambda\in\Y_n\left(\widehat{G}\right)}a\left(\Lambda\right)\chi^{\Lambda}(x),
\;\; x\in S_n(G).
\end{equation}
Using (\ref{12.6.6}) and (\ref{12.6.7}) we find
\begin{equation}
\chi_z\vert_{S_n(G)}(w)=\frac{n!}{(I)_n}
\sum\limits_{\Lambda,\Lambda'\in\Y_n\left(\widehat{G}\right)}a\left(\Lambda'\right)\overline{a\left(\Lambda\right)}
\int\limits_{S_n(G)}\chi^{\Lambda'}(xw)\overline{\chi^{\Lambda}(x)}d\mu_{S_n(G)}(x).
\end{equation}
The orthogonality of the irreducible characters of $S_n(G)$ is expressed by the formula
\begin{equation}\label{12.6.8}
\int\limits_{S_n(G)}\chi^{\Lambda'}(xw)\overline{\chi^{\Lambda}(x)}d\mu_{S_n(G)}(x)
=\frac{\chi^{\Lambda}(w)}{\DIM\Lambda}\delta_{\Lambda,\Lambda'}.
\end{equation}
We thus arrive to the following formula for the characters $\chi_z$:
\begin{equation}\label{12.6.9}
\chi_z\vert_{S_n(G)}(w)=\frac{n!}{(I)_n}
\sum\limits_{\Lambda\in\Y_n\left(\widehat{G}\right)}
|a\left(\Lambda\right)|^2\frac{\chi^{\Lambda}(w)}{\DIM\Lambda},
\end{equation}
where $a\left(\Lambda\right)$ is the coefficient in the decomposition (\ref{12.6.7}).
\section{Proof of Theorem \ref{Theorem13.2.2}}\label{SECTIONPROOFT1}
From equations   (\ref{12.6.9}) and (\ref{13.1.2}) we obtain the following formula
\begin{equation}\label{13.2.1}
M_z^{(n)}\left(\Lambda\right)=\frac{n!}{\left(I\right)_n}|a(\Lambda)|^2,
\end{equation}
where $a(\Lambda)$ is defined by equation (\ref{12.6.7}), and $I=\int\limits_G|z(g)|^2d\mu_G(g)$.
Given $\Lambda\in\Y_n\left(\widehat{G}\right)$, we need to compute
\begin{equation}\label{ch1.1.2}
a(\Lambda)=\left<\varphi_z,\chi^{\Lambda}\right>_{S_n(G)},
\end{equation}
where $\varphi_z$ is defined by equation (\ref{ch1.1.1}). For this purpose, we will consider the inner product
$\left<\Psi,\chi^{\Lambda}\right>_{S_n(G)}$, where $\Psi(x)$
is a central function on $S_n(G)$ defined in terms of Newton
power symmetric functions; see equation (\ref{PSI}) below.
 We will explain in this Section that Theorem \ref{Theorem13.2.2} follows from Theorem \ref{Theoremch1.3.2}, which explicitly determines this inner product.  Theorem \ref{Theoremch1.3.2} is of independent interest and will be proved in Section \ref{SECTIONPROOFTHEOREMCH}.
\subsection{The characters $\chi^{\Lambda}$ and symmetric functions}
Let $p_1$, $p_2$, $\ldots$ be central complex-valued functions defined on $G$.
For each $\zeta\in\widehat{G}$ define
\begin{equation}\label{ch1.3.3}
\widehat{p}_r(\zeta)=\left<p_r,\chi^{\pi^{\zeta}}\right>_G,
\end{equation}
where $r=1,2,\ldots$, and where $\chi^{\pi^{\zeta}}$ is the character of an irreducible representation $\pi^{\zeta}$ of $G$ parameterized by $\zeta\in \widehat{G}$.
Consider $\widehat{p}_1(\zeta)$, $\widehat{p}_2(\zeta)$, $\ldots$ as independent indeterminates over $\C$, and let
\begin{equation}
\widehat{\SYM}(\zeta)=\C\left[\widehat{p}_r(\zeta),\;r\geq 1\right].
\end{equation}
The algebra $\widehat{\SYM}(\zeta)$ is generated by $\widehat{p}_1(\zeta)$, $\widehat{p}_2(\zeta)$, $\ldots$.
Given $\lambda=(\lambda_1,\lambda_2,\ldots)\in \Y$, the power symmetric function $\widehat{p}_{\lambda}(\zeta)$ is defined by
$$
\widehat{p}_{\lambda}(\zeta)=\widehat{p}_{\lambda_1}(\zeta)\widehat{p}_{\lambda_2}(\zeta)\ldots.
$$
The functions $\widehat{p}_{\lambda}(\zeta)$ form a $\C$-basis in $\widehat{\SYM}(\zeta)$, and $\widehat{\SYM}(\zeta)$ can be understood as an algebra of symmetric functions generated by $\widehat{p}_1(\zeta)$, $\widehat{p}_2(\zeta)$, $\ldots$.
The Schur functions $\widehat{s}_{\lambda}(\zeta)$ are defined in terms of $\widehat{p}_{\lambda}(\zeta)$ using the Frobenius formula,
\begin{equation}\label{SF75}
\widehat{s}_{\lambda}(\zeta)=\sum\limits_{\nu\in\Y_{|\lambda|}}\frac{1}{z_{\nu}}
\chi^{\lambda}_{\nu}\widehat{p}_{\nu}(\zeta),
\end{equation}
where
$$
z_{\nu}=\prod\limits_{i=1}^{\infty}i^{m_i(\nu)}m_i(\nu)!,\;\; \nu=\left(1^{m_1(\nu)}2^{m_2(\nu)}\ldots\right)\in\Y_{|\lambda|},
$$
and where $\chi^{\lambda}_{\nu}$ is the value of the character of the irreducible representation of the symmetric group $S_{|\lambda|}$ parameterized by $\lambda$ on the conjugacy class defined by $\nu$.

Let $\varrho\in\Y_n\left([G]\right)$ parametrize the conjugacy class
of $x\in S_n(G)$. The family $\varrho$ can be understood as a map
$$
\varrho:\;[G]\longrightarrow\Y,\;
\sum\limits_{c\in [G]}|\varrho(c)|=n.
$$
For each $c\in[G]$ write
$$
\varrho(c)=1^{m_1(\varrho(c))}2^{m_2(\varrho(c))}\ldots n^{m_n(\varrho(c)}.
$$
Set
\begin{equation}\label{PSI}
\Psi(x)=\prod\limits_{c\in[G]}p_{\varrho(c)}\left(p_1(c),p_2(c),\ldots\right),\;\; x\in S_n(G),
\end{equation}
where $p_1$, $p_2$, $\ldots$ are the given central functions on $G$,
$\varrho$ determines the conjugacy class of $x\in S_n(G)$, and the Newton power symmetric function $p_{\varrho(c)}$ is defined in terms of $p_1(c)$, $p_2(c)$, $\ldots$
by
\begin{equation}\label{ch1.2.1}
p_{\varrho(c)}\left(p_1(c),p_2(c),\ldots\right)=(p_1(c))^{m_1(\varrho(c))}(p_2(c))^{m_2(\varrho(c))}\ldots (p_n(c))^{m_n(\varrho(c))}.
\end{equation}
\begin{thm}\label{Theoremch1.3.2}
Denote by $\widehat{s}_{\lambda}(\zeta)$ the Schur function in
$\widehat{\SYM}(\zeta)$ parameterized by the Young diagram $\lambda$ and defined by equations (\ref{ch1.3.3}), (\ref{SF75}).
We have
\begin{equation}\label{ch1.3.4}
\left<\Psi,\chi^{\Lambda}\right>_{S_n(G)}=\prod\limits_{\zeta\in\widehat{G}}\widehat{s}_{\Lambda(\zeta)}(\zeta),
\end{equation}
where $\chi^{\Lambda}$ is the character of the irreducible representation of $S_n(G)$ parameterized by the map
$$
\Lambda:\;\widehat{G}\longrightarrow\Y,\;\;\sum\limits_{\zeta\in\widehat{G}}|\Lambda(\zeta)|=n,
$$
and where $\Psi$ is defined by equation (\ref{PSI}).
\end{thm}
We will prove Theorem \ref{Theoremch1.3.2} in Section \ref{SECTIONPROOFTHEOREMCH}.
\subsection{The formula for the $z$-measures}
Here we would like to show how  formula (\ref{13.2.2.2}) for the $z$-measures can be obtained from Theorem \ref{Theoremch1.3.2}. This will prove Theorem \ref{Theorem13.2.2}  under the assumption that
Theorem \ref{Theoremch1.3.2} holds true.

Let all $p_1(c)$, $p_2(c)$, $\ldots$ be equal to the function $z(c)$. If
$\varrho:\;[G]\longrightarrow\C$, $\sum\limits_{c\in[G]}|\varrho(c)|=n$
describes the conjugacy class of $x\in S_n(G)$, then
\begin{equation}\label{ch2.1.1}
\begin{split}
&p_{\varrho(c)}=(p_1(c))^{m_1(\varrho(c))}
(p_2(c))^{m_2(\varrho(c))}\ldots(p_n(c))^{m_n(\varrho(c))}\\
&=\left(z(c)\right)^{m_1(\varrho(c))+m_2(\varrho(c))+\ldots+m_n(\varrho(c))}
=(z(c))^{[x](c)}.
\end{split}
\end{equation}
Therefore, the function $\Psi$ defined by (\ref{PSI}) turns into
\begin{equation}\label{ch2.1.2}
\begin{split}
\Psi(x)
=\prod\limits_{c\in[G]}(z(c))^{[x](c)},
\end{split}
\end{equation}
and
$$
a(\Lambda)=\left<\Psi,\chi^{\Lambda}\right>_{S_n{G}},
$$
see equations (\ref{ch1.1.1}), (\ref{ch1.1.2}),
and (\ref{ch2.1.2}). Formula (\ref{ch1.3.4}) gives
\begin{equation}\label{ch2.1.3}
\begin{split}
a(\Lambda)=\prod\limits_{\zeta\in\widehat{G}}\widehat{s}_{\Lambda(\zeta)}(\zeta),
\end{split}
\end{equation}
where $\widehat{s}_{\Lambda(\zeta)}(\zeta)$ is the Schur function in $\widehat{\SYM}(\zeta)$
parameterized by the Young diagram $\Lambda(\zeta)$. For the algebra $\widehat{\SYM}(\zeta)$,
it is generated by $\widehat{p}_1(\zeta)$, $\widehat{p}_2(\zeta)$, $\ldots$ defined by equation (\ref{ch1.3.3}). We find
\begin{equation}\label{ch2.1.4}
\begin{split}
\widehat{p}_r(\zeta)=\left<z,\chi^{\pi^\zeta}\right>_G,\; r=1,2,\ldots.
\end{split}
\end{equation}
It is known that if in the algebra $\SYM=\C[p_1,p_2,\ldots]$ all the generators
$p_1$, $p_2$, $\ldots$ are equal to a complex number $\alpha$, then the corresponding
Schur function $s_{\lambda}$ is equal to
\begin{equation}\label{ch2.1.5}
\begin{split}
s_{\lambda}
=\prod\limits_{\Box\in\lambda}\frac{\alpha+c(\Box)}{h(\Box)},\;\; \forall \lambda\in\Y,
\end{split}
\end{equation}
see, for example, Macdonald \cite{Macdonald}, Section I.3, Example 4. Therefore, if $\alpha(\zeta)$
is given by equation (\ref{13.2.2.3}), then (\ref{ch2.1.3}), (\ref{ch2.1.4}), and
(\ref{ch2.1.5}) lead to the formula
\begin{equation}\label{ch2.1.6}
\begin{split}
a(\Lambda)
=\prod\limits_{\zeta\in\hat{G}}\prod\limits_{\Box\in\Lambda(\zeta)}\frac{\alpha(\zeta)+c(\Box)}{h(\Box)}.
\end{split}
\end{equation}
Equation (\ref{ch2.1.6}) together with equation (\ref{13.2.1}) give us the formula for the $z$-measures, equation (\ref{13.2.2.2}).

\begin{rem}
We wish to demonstrate by direct computation that $M_z^{(n)}$ given by equation
(\ref{13.2.2.2}) is a probability measure on $\Y_n(\widehat{G})$.
Note that $M_z^{(n)}(\Lambda)$ can be written as
\begin{equation}\label{6.2.1}
M_z^{(n)}\left(\Lambda\right)=\frac{n!}{(I)_n}\prod\limits_{\zeta\in\widehat{G}}
\frac{\left(\alpha(\zeta)
\overline{\alpha(\zeta)}\right)_{|\Lambda(\zeta)|}}{|\Lambda(\zeta)|!}\;
\mathfrak{m}_{\alpha(\zeta)}^{\left(|\Lambda(\zeta)|\right)}
\left(\Lambda(\zeta)\right),
\end{equation}
where
\begin{equation}\label{6.2.2}
\mathfrak{m}_{a}^{(k)}(\lambda)=\frac{k!}{(a\bar{a})_k}\prod\limits_{\Box\in\lambda}
\frac{(a+c(\Box))(\bar{a}+c(\Box))}{h^2(\Box)}
\end{equation}
is the $z$-measure with the parameter $a$ on $\Y_k$  for the symmetric group $S_k$, see, for example, Borodin and Olshanski \cite{BorodinOlshanskiBook}, Proposition 11.8.

First, let us show that
\begin{equation}\label{6.4.1}
\sum\limits_{\Lambda\in\Y_n(\widehat{G})}\left[\prod\limits_{\zeta\in\widehat{G}}
\frac{\left(\alpha(\zeta)\overline{\alpha(\zeta)}\right)_{|\Lambda(\zeta)|}}{|\Lambda(\zeta)|!}\;\mathfrak{m}_{\alpha(\zeta)}^{\left(|\Lambda(\zeta)|\right)}
\left(\Lambda(\zeta)\right)\right]=\frac{\left(\sum\limits_{\zeta\in\widehat{G}}
\alpha(\zeta)\overline{\alpha(\zeta)}\right)_n}{n!}.
\end{equation}
Set
\begin{equation}
\mathcal{P}_n=\left\{p: \widehat{G}\longrightarrow\mathbb{Z}_{\geq 0},\;\; \sum\limits_{\zeta\in\widehat{G}}p(\zeta)=n\right\}.
\end{equation}
The left-hand side of equation (\ref{6.4.1}) is equal to
\begin{equation}
\begin{split}
&\sum\limits_{p\in\mathcal{P}_n}
\underset{|\Lambda(\eta)|=p(\eta),\; \forall\eta\in\widehat{G}.}{\sum\limits_{\Lambda\in\Y_n\left(\widehat{G}\right)}}
\prod\limits_{\zeta\in\widehat{G}}
\frac{\left(\alpha(\zeta)
\overline{\alpha(\zeta)}\right)_{p(\zeta)}}{p(\zeta)!}\;
\mathfrak{m}_{\alpha(\zeta)}^{(p(\zeta))}
\left(\Lambda(\zeta)\right)\\
&=\sum\limits_{p\in\mathcal{P}_n}
\left[\prod\limits_{\zeta\in\widehat{G}}
\frac{\left(\alpha(\zeta)
\overline{\alpha(\zeta)}\right)_{p(\zeta)}}{p(\zeta)!}\right]
\left[\underset{|\Lambda(\eta)|=p(\eta),\; \forall\eta\in\widehat{G}.}{\sum\limits_{\Lambda\in\Y_n\left(\widehat{G}\right)}}
\prod\limits_{\zeta\in\widehat{G}}\mathfrak{m}_{\alpha(\zeta)}^{(p(\zeta))}
\left(\Lambda(\zeta)\right)\right].
\end{split}
\end{equation}
Let us show that the expression in the second bracket (on the right-hand side of the equation written above) is equal to $1$, for each $p\in\mathcal{P}_n$.
In fact, assume that
\begin{equation}
p(\eta)=\left\{
  \begin{array}{ll}
    n_1, & \eta=\zeta_1,\\
    n_2, & \eta=\zeta_2, \\
    \vdots, &  \\
    n_r, & \eta=\zeta_r,\\
    0, & \mbox{otherwise},
  \end{array}
\right.
\end{equation}
where $n_1+n_2+\ldots+n_r=n$. Then each $\Lambda$ in the sum in the second bracket
can be represented as $\Lambda=\left(\lambda^{(1)},\ldots,\lambda^{(r)}\right)$, where
$\lambda^{(1)}$, $\ldots$, $\lambda^{(r)}$ are Young diagrams such that
$|\lambda^{(1)}|=n_1$, $\ldots$, $|\lambda^{(r)}|=n_r$. The  sum can be written as
\begin{equation}
\begin{split}
&\underset{|\Lambda(\eta)|=p(\eta),\; \forall\eta\in\widehat{G}.}{\sum\limits_{\Lambda\in\Y_n\left(\widehat{G}\right)}}
\prod\limits_{\zeta\in\widehat{G}}\mathfrak{m}_{\alpha(\zeta)}^{(p(\zeta))}
\left(\Lambda(\zeta)\right)\\
&=\sum\limits_{\lambda^{(1)}\in\Y_{n_1}}\ldots\sum\limits_{\lambda^{(r)}\in\Y_{n_r}}\mathfrak{m}_{\alpha(\zeta_1)}^{(n_1)}(\lambda^{(1)})\ldots
\mathfrak{m}_{\alpha(\zeta_r)}^{(n_r)}(\lambda^{(r)})=1.
\end{split}
\end{equation}
Second, we observe that
\begin{equation}\label{6.4.4}
\sum\limits_{p\in\mathcal{P}_n}
\prod\limits_{\zeta\in\widehat{G}}
\frac{\left(\alpha(\zeta)
\overline{\alpha(\zeta)}\right)_{p(\zeta)}}{p(\zeta)!}
=\frac{\left(\sum\limits_{\zeta\in\widehat{G}}\alpha(\zeta)\overline{\alpha(\zeta)}\right)_n}{n!}.
\end{equation}
The formula above follows from
\begin{equation}\label{6.4.5}
\underset{m_1+\ldots+m_k=n}{\sum\limits_{m_1,\ldots,m_k\in\left\{0,1,\ldots\right\}}}
\frac{(a_1)_{m_1}(a_2)_{m_2}\ldots (a_k)_{m_k}}{m_1!m_2!\ldots m_k!}=\frac{(a_1+\ldots+a_k)_n}{n!},
\end{equation}
and from the fact that (\ref{6.4.5}) holds true for each $k$, $k=1,2,\ldots$.
\end{rem}

\section{The Hirai, Hirai, and Hora construction and the proof of Theorem \ref{Theoremch1.3.2}}\label{SECTIONPROOFTHEOREMCH}
Our proof of Theorem \ref{Theoremch1.3.2} is based on the Hirai, Hirai, and Hora construction of irreducible representations of $S_n(G)$, which we will describe in the following.

\subsection{The Hirai, Hirai, and Hora construction of irreducible
representations of $S_n(G)$}
In the case of a compact group $G$ irreducible unitary representations of the wreath product $S_n(G)$ are constructed by Hirai, Hirai, and Hora in Ref. \cite{HiraiHiraiHoraI}.
We also refer the reader to Hora, Hirai, Hirai \cite{HoraHiraiHiraiII},
to Hora and Hirai \cite{HoraHirai},  and to the lectures by Hora \cite{HoraLectures}, where
some aspects of work \cite{HiraiHiraiHoraI}
are explained in detail.

Here we give a short summary of the constriction in Ref. \cite{HiraiHiraiHoraI}. Assume that $\Lambda=\left(\lambda^{\zeta}\right)_{\zeta\in\widehat{G}}$ is such that $\sum_{\zeta\in\widehat{G}}|\lambda^{\zeta}|=n$.
We wish to obtain an irreducible unitary representation $\pi^{\Lambda}$ of $S_n(G)$ parameterized by $\Lambda$.

Consider a partition of $\{1,2,\ldots,n\}$ into blocks of size $\left|\lambda^{\zeta}\right|$:
\begin{equation}\label{HHH3.2.2}
\{1,2,\ldots,n\}=\underset{\zeta\in\widehat{G}}\bigcup I_{n,\zeta},\;\; \left|I_{n,\zeta}\right|=\left|\lambda^{\zeta}\right|.
\end{equation}
The subgroup $\prod\limits_{\zeta\in\widehat{G}}S_{I_{n,\zeta}}$
(where $S_{I_{n,\zeta}}$ is the permutation group of $I_{n,\zeta}$) is isomorphic to $\prod\limits_{\zeta\in\widehat{G}}S_{|\lambda^{\zeta}|}$
(where $S_{|\lambda^{\zeta}|}$ is the permutation group of
$
\left\{1,2,\ldots,\left|\lambda^\zeta\right|\right\}.)
$

Given a partition of $\{1,2,\ldots,n\}$ as in (\ref{HHH3.2.2}) we introduce $\eta\in \widehat{D_n(G)}$ (see the notation in Section \ref{SubsectionDn}) as a list of those $\zeta\in\widehat{G}$ for which $|\lambda^{\zeta}|\neq0$. That is, if $\zeta\in\widehat{G}$ is such that $|\lambda^{\zeta}|\neq 0$, and  $I_{n,\zeta}=\{i_1,\ldots,i_k\}$,  then this specific $\zeta$ occupies the places numbered $i_1$, $\ldots$, $i_k$ in the list $\eta$. It is important that the group $S^{\eta}$ defined by (\ref{reps5})
is isomorphic to $\prod\limits_{\zeta\in\widehat{G}}S_{I_n,\zeta}$, and this is isomorphic to
$\prod\limits_{\zeta\in\widehat{G}}S_{|\lambda^{\zeta}|}$.
Set
\begin{equation}\label{HHH3.4.1}
H_n=G^n\rtimes S^{\eta}\cong
\prod\limits_{\zeta\in\widehat{G}}S_{|\lambda^{\zeta}|}(G).
\end{equation}
Clearly, $H_n$ is a subgroup of $S_n(G)$.

Let $\left(\pi^{\zeta},V^{\zeta}\right)$ be an irreducible unitary representation of $G$.
Assume that $\zeta\in\widehat{G}$ is such that $|\lambda^{\zeta}|\neq 0$.
A representation $\varrho^{\zeta}$ of $S_{|\lambda^{\zeta}|}(G)$ is defined by
\begin{equation}\label{HHH3.5.1}
\varrho^{\zeta}(g,s)=\left(\pi^{\zeta}\right)^{\boxtimes\; |\lambda^{\zeta}|}
(g)I(s),\;\;(g,s)\in S_{|\lambda^{\zeta}|}(G).
\end{equation}
The representation space of $\varrho^{\zeta}$ is $\left(V^{\zeta}\right)^{\otimes\;|\lambda^{\zeta}|}$, and the action of $I(s)$ on this space is defined by
\begin{equation}\label{HHH3.5.2}
I(s)\left(v_1\otimes\ldots\otimes v_{|\lambda^{\zeta}|}\right)=
v_{s^{-1}(1)}\otimes\ldots\otimes v_{s^{-1}(|\lambda^{\zeta}|)},
\end{equation}
where $s\in S_{|\lambda^{\zeta}|}$. It is not difficult to see that $\varrho^{\zeta}$ is an irreducible unitary representation of $S_{|\lambda^{\zeta}|}(G)$.
Since $H_n$ defined by (\ref{HHH3.4.1}) is isomorphic to $\prod\limits_{\zeta\in\widehat{G}}S_{|\lambda^{\zeta}|}(G)$, an irreducible representation
$\varrho^{\eta}$ of $H_n$ can be defined in terms of the representations
$\varrho^{\zeta}$ of $S_{|\lambda^{\zeta}|}(G)$,
\begin{equation}
\varrho^{\eta}=\underset{\zeta\in\widehat{G},\; |\lambda^{\zeta}|\neq 0}{\boxtimes}\varrho^{\zeta}.
\end{equation}
The representation space of
$\varrho^{\eta}$ is
$$
V^{\eta}=\underset{\zeta\in\widehat{G},\; |\lambda^{\zeta}|\neq 0}{\otimes}
\left(V^{\zeta}\right)^{\otimes\;|\lambda^{\zeta}|}.
$$
In addition, set
$$
\xi^{\eta}=\underset{\zeta\in\widehat{G},\; |\lambda^{\zeta}|\neq 0}{\boxtimes}\pi^{\lambda^\zeta},
$$
where $\pi^{\lambda^{\zeta}}$ denotes the irreducible representation of $S_{|\lambda^{\zeta}|}$ parameterized by the Young diagram $\lambda^{\zeta}$. The representation space of $\xi^{\eta}$ is $\underset{\zeta\in\widehat{G},\; |\lambda^{\zeta}|\neq 0}{\otimes}
V^{\lambda^{\zeta}}$, where $V^{\lambda^{\zeta}}$ denotes the representation space of $\pi^{\lambda^{\zeta}}$. Clearly, $\xi^{\eta}$ can be understood as an irreducible representation of $S^{\eta}$. In what follows, we agree that $\xi^{\eta}$ acts as a trivial representation of the subgroup $G^n$ of $H_n$, and we regard $\xi^{\eta}$ as an irreducible representation of $H_n$.

In this way, we obtain an irreducible unitary representation $\varrho^{\eta}\otimes\xi^{\eta}$ of the subgroup $H_n$ of $S_n(G)$.
\begin{thm}\label{TheoremHHH3.8.1}
For $\Lambda=\left(\lambda^{\zeta}\right)_{\zeta\in\widehat{G}}\in\Y_n\left(\widehat{G}\right)$ set
\begin{equation}\label{HHH3.8.2}
\pi^{\Lambda}=\Ind_{H_n}^{S_n(G)}\left(\varrho^{\eta}\otimes\xi^{\eta}\right).
\end{equation}
Then $\pi^{\Lambda}$ is an irreducible unitary representation of $S_n(G)$.
Moreover, any irreducible unitary representation of $S_n(G)$ is equivalent to one of the induced representations $\pi^{\Lambda}$ defined by (\ref{HHH3.8.2}).
\end{thm}
\begin{proof}See Hirai, Hirai, and Hora \cite{HiraiHiraiHoraI}, Section 3.2.
\end{proof}
\subsection{A formula for the character  $\chi^{\varrho^{\eta}\otimes\xi^{\eta}}$}
\begin{prop}\label{chchiProposition5.1.5}
Let $\left(\pi^{\zeta},V^{\zeta}\right)$ denote the irreducible unitary representation of
a compact group $G$ parameterized by $\zeta\in\widehat{G}$. Assume that
\begin{equation}\label{chchi5.1.2}
s=\left(i_1i_2\ldots i_{l_1}\right)\left(j_1j_2\ldots j_{l_2}\right)\ldots\left(\nu_1\nu_2\ldots\nu_{l_p}\right)\in S_{|\lambda^{\zeta}|},
\end{equation}
so $l_1+l_2+\ldots+l_p=|\lambda^{\zeta}|$.
Then we have
\begin{equation}\label{chchi5.1.3}
\begin{split}
&\chi^{\varrho^{\zeta}}\left[\left((g_1,\ldots,g_{|\lambda^{\zeta}|}),s\right)\right]\\
&=\chi^{\pi^{\zeta}}\left(g_{i_{l_1}}\ldots g_{i_2}g_{i_1}\right)
\chi^{\pi^{\zeta}}\left(g_{j_{l_2}}\ldots g_{j_2}g_{j_1}\right)
\ldots \chi^{\pi^{\zeta}}\left(g_{\nu_{l_p}}\ldots g_{\nu_2}g_{\nu_1}\right),
\end{split}
\end{equation}
where $\chi^{\pi^{\zeta}}$ denotes the character of the irreducible representation $\left(\pi^{\zeta},V^{\zeta}\right)$
of $G$ (see Section \ref{NotSection1.2}),
and $\chi^{\varrho^{\zeta}}$ denotes the character of the representation
$\left(\varrho^{\zeta},\left(V^{\zeta}\right)^{\otimes |\lambda^{\zeta}|}\right)$ of the group
$S_{|\lambda^{\zeta}|}(G)$ defined by (\ref{HHH3.5.1}).
\end{prop}
\begin{proof}
Suppose that $\dim V^{\zeta}=m$, and that $\mathcal{B}_{V^{\zeta}}=\left(v_1,\ldots,v_m\right)$
is a basis of $V^{\zeta}$. A suitable basis for $\left(V^{\zeta}\right)^{\otimes|\lambda^{\zeta}|}$
is formed by $v_{a_1}\otimes v_{a_2}\otimes\ldots\otimes v_{a_k}$, where $k=|\lambda^{\zeta}|$, and $a_1,\ldots,a_k$ take values in $\{1,2,\ldots,m\}$. We can write
\begin{equation}
\begin{split}
&\left(\pi^{\zeta}(g_1)\boxtimes\pi^{\zeta}(g_2)\boxtimes\ldots\boxtimes\pi^{\zeta}(g_k)\right)
I(s)\left(v_{a_1}\otimes v_{a_2}\otimes\ldots\otimes v_{a_k}\right)\\
&=\left(\pi^{\zeta}(g_1)\boxtimes\pi^{\zeta}(g_2)\boxtimes\ldots\boxtimes\pi^{\zeta}(g_k)\right)
\left(v_{a_{s^{-1}(1)}}\otimes v_{a_{s^{-1}(2)}}\otimes\ldots\otimes v_{a_{s^{-1}(k)}}\right)\\
&=\sum\limits_{b_1,\ldots,b_k=1}^m
A^{(1)}_{b_1,a_{s^{-1}(1)}}A^{(2)}_{b_2,a_{s^{-1}(2)}}\ldots A^{(k)}_{b_k,a_{s^{-1}(k)}}
v_{b_1}\otimes v_{b_2}\otimes\ldots\otimes v_{b_k},
\end{split}
\end{equation}
where $A^{(1)}$ is the matrix of $\pi^{\zeta}(g_1)$ in the basis $\mathcal{B}_{V^{\zeta}}$,
$\ldots$, $A^{(k)}$ is the matrix of $\pi^{\zeta}(g_k)$ in the basis $\mathcal{B}_{V^{\zeta}}$.
This gives the following formula for $\chi^{\varrho^{\zeta}}$:
\begin{equation}
\chi^{\varrho^{\zeta}}\left[\left((g_1,\ldots,g_{|\lambda^{\zeta}|}),s\right)\right]
=\sum\limits_{a_1,\ldots,a_k=1}^m
A^{(1)}_{a_1,a_{s^{-1}(1)}}A^{(2)}_{a_2,a_{s^{-1}(2)}}\ldots A^{(k)}_{a_k,a_{s^{-1}(k)}},\;\; k=|\lambda^{\zeta}|.
\end{equation}
Next, we use the explicit form of $s$ (equation (\ref{chchi5.1.2})), and the rewrite the right-hand side of the equation above as
\begin{equation}
\begin{split}
&\sum\limits_{a_1,\ldots,a_k=1}^m
\left[A^{(i_{l_1})}_{a_{i_{l_1}},a_{i_{l_1-1}}}\ldots A^{(i_2)}_{a_{i_2},a_{i_1}} A^{(i_1)}_{a_{i_1},a_{i_{l_1}}}\right]
\left[A^{(j_{l_2})}_{a_{j_{l_2}},a_{j_{l_2-1}}}\ldots A^{(j_2)}_{a_{j_2},a_{j_1}} A^{(j_1)}_{a_{j_1},a_{j_{l_2}}}\right]\ldots\\
&\times
\left[A^{(\nu_{l_p})}_{a_{nu_{l_p}},a_{\nu_{l_p-1}}}\ldots A^{(\nu_2)}_{a_{\nu_2},a_{\nu_1}} A^{(\nu_1)}_{a_{\nu_1},a_{\nu_{l_p}}}\right]\\
&=\Tr\left[A^{(i_{l_1})}\ldots A^{(i_2)} A^{(i_1)}\right]
\Tr\left[A^{(j_{l_2})}\ldots A^{(j_2)} A^{(j_1)}\right]\ldots
\Tr\left[A^{(\nu_{l_p})}\ldots A^{(\nu_2)} A^{(\nu_1)}\right].
\end{split}
\end{equation}
Thus we obtain
\begin{equation}
\begin{split}
&\chi^{\varrho^{\zeta}}\left[\left(\left(g_1,\ldots,g_k\right),s\right)\right]=\Tr\left[\pi^{\zeta}\left(g_{i_{l_1}}\right)
\ldots \pi^{\zeta}\left(g_{i_{2}}\right)\pi^{\zeta}\left(g_{i_{1}}\right)\right]
\\
&\times\Tr\left[\pi^{\zeta}\left(g_{j_{l_2}}\right)
\ldots \pi^{\zeta}\left(g_{j_{2}}\right)\pi^{\zeta}\left(g_{j_{1}}\right)\right]
\ldots
\Tr\left[\pi^{\zeta}\left(g_{\nu_{l_p}}\right)
\ldots \pi^{\zeta}\left(g_{\nu_{2}}\right)\pi^{\zeta}\left(g_{\nu_{1}}\right)\right]\\
&=\Tr\left[\pi^{\zeta}\left(g_{i_{l_1}}
\ldots g_{i_{2}}g_{i_{1}}\right)\right]
\Tr\left[\pi^{\zeta}\left(g_{j_{l_2}}
\ldots g_{j_{2}}g_{j_{1}}\right)\right]
\ldots
\Tr\left[\pi^{\zeta}\left(g_{\nu_{l_p}}
\ldots g_{\nu_{2}}g_{\nu_{1}}\right)\right],
\end{split}
\end{equation}
which is the right-hand side of equation (\ref{chchi5.1.3}).
\end{proof}
Assume that
$
\widehat{G}=\left\{\zeta_q\right\}_{q\in\mathcal{K}},
$
where $\mathcal{K}$ is a countable set. Then equivalence classes of the irreducible unitary representations of $S_n(G)$ are parameterized by the maps
\begin{equation}\label{chchi5.2.1}
\Lambda:\;\widehat{G}\longrightarrow\Y,\;\;
\sum\limits_{q\in\mathcal{K}}|\Lambda(\zeta_q)|=n.
\end{equation}
Suppose that a such map  $\Lambda$ is given, and it is defined by
\begin{equation}\label{chchi5.2.2}
\Lambda\left(\zeta_q\right)=
\left\{
  \begin{array}{ll}
    \lambda^{\zeta_{\tau_1}}, & q=\tau_1, \\
    \lambda^{\zeta_{\tau_2}}, & q=\tau_2, \\
    \vdots &  \\
    \lambda^{\zeta_{\tau_r}}, & q=\tau_r, \\
    \emptyset, & \hbox{otherwise.}
  \end{array}
\right.
\end{equation}
Set
$$
n_1=|\lambda^{\zeta_{\tau_1}}|,\ldots, n_r=|\lambda^{\zeta_{\tau_r}}|.
$$
We have $n_1+\ldots+n_r=n$, and the subgroup $H_n$ of $S_n(G)$ is isomorphic to
$$
S_{n_1}(G)\times\ldots\times S_{n_r}(G).
$$
Assume that
$$
\left(s^{(1)},\ldots, s^{(r)}\right)\in S_{n_1}(G)\times\ldots\times S_{n_r}(G),
$$
and write
\begin{equation}
\begin{split}
&s^{(1)}=\left(i^{(1)}_1i^{(1)}_2\ldots i^{(1)}_{l_1^{(1)}}\right)
\left(j^{(1)}_1j^{(1)}_2\ldots j^{(1)}_{l_2^{(1)}}\right)
\ldots
\left(\nu^{(1)}_1\nu^{(1)}_2\ldots \nu^{(1)}_{l_{n_1}^{(1)}}\right),\\
&\vdots\\
&s^{(r)}=\left(i^{(r)}_1i^{(r)}_2\ldots i^{(r)}_{l_1^{(r)}}\right)
\left(j^{(r)}_1j^{(r)}_2\ldots j^{(r)}_{l_2^{(r)}}\right)
\ldots
\left(\nu^{(r)}_1\nu^{(r)}_2\ldots \nu^{(r)}_{l_{n_r}^{(r)}}\right).
\end{split}
\end{equation}
Note that the conditions
\begin{equation}
l_1^{(1)}+l_2^{(1)}+\ldots+l_{n_1}^{(1)}=n_1,\ldots,
l_1^{(r)}+l_2^{(r)}+\ldots+l_{n_r}^{(r)}=n_r
\end{equation}
are satisfied. Since
$
\varrho^{\eta}=\boxtimes_{\zeta\in\widehat{G}}\varrho^{\zeta},
$
we obtain
\begin{equation}\label{chchi5.3.2}
\begin{split}
&\chi^{\varrho^{\eta}}\left[\left(\left(g_1^{(1)},\ldots,g_{n_1}^{(1)}\right)s^{(1)}\right)\ldots
\left(\left(g_1^{(r)},\ldots,g_{n_r}^{(r)}\right)s^{(r)}\right)\right]\\
&=\chi^{\varrho^{\zeta_{\tau_1}}}\left[\left(\left(g_1^{(1)},\ldots,g_{n_1}^{(1)}\right)s^{(1)}\right)\right]
\ldots
\chi^{\varrho^{\zeta_{\tau_r}}}\left[\left(\left(g_1^{(r)},\ldots,g_{n_r}^{(r)}\right)s^{(r)}\right)\right],
\end{split}
\end{equation}
and the characters in the right-hand side of equation (\ref{chchi5.3.2}) are given by equation (\ref{chchi5.1.3}). Finally, we obtain
\begin{equation}\label{chchi5.3.1}
\begin{split}
&\chi^{\varrho^{\eta}\boxtimes\xi^{\eta}}\left[\left(\left(g_1^{(1)},\ldots,g_{n_1}^{(1)}\right)s^{(1)}\right)\ldots
\left(\left(g_1^{(r)},\ldots,g_{n_r}^{(r)}\right)s^{(r)}\right)\right]\\
&=\chi^{\varrho^{\eta}}\left[\left(\left(g_1^{(1)},\ldots,g_{n_1}^{(1)}\right)s^{(1)}\right)\ldots
\left(\left(g_1^{(r)},\ldots,g_{n_r}^{(r)}\right)s^{(r)}\right)\right]
\chi^{\lambda^{\zeta_{\tau_1}}}(s^{(1)})\ldots\chi^{\lambda^{\zeta_{\tau_r}}}(s^{(r)}).
\end{split}
\end{equation}
Equations (\ref{chchi5.3.1}), (\ref{chchi5.3.2}), and (\ref{chchi5.1.3}) give us the desired characters $\chi^{\varrho^{\eta}\otimes\xi^{\eta}}$
of the representation $\varrho^{\eta}\otimes\xi^{\eta}$ explicitly,
in terms of the characters $\chi^{\pi^{\zeta_{\tau_1}}}$, $\ldots$, $\chi^{\pi^{\zeta_{\tau_r}}}$ of the irreducible representations $\pi^{\zeta_{\tau_1}}$, $\ldots$, $\pi^{\zeta_{\tau_r}}$ of $G$, and in terms of  the characters $\chi^{\lambda^{\zeta_{\tau_1}}}$, $\ldots$, $\chi^{\lambda^{\zeta_{\tau_r}}}$ of the irreducible representations of symmetric groups $S_{n_1}$, $\ldots$, $S_{n_r}$.
\subsection{Proof of Theorem \ref{Theoremch1.3.2}}
Assume that the map $\Lambda: \widehat{G}\longrightarrow\Y$ that defines an irreducible representation $\pi^{\Lambda}$ is given by equation (\ref{chchi5.2.2}).
Recall that $H_n$ is defined by (\ref{HHH3.4.1}). Since $H_n$ is isomorphic to
$$
S_{|\lambda^{\zeta_{\tau_1}}|}(G)\times\ldots\times S_{|\lambda^{\zeta_{\tau_\tau}}|}(G),
$$
the value of the function $\Psi$  defined by equation (\ref{PSI}) on a group element $x$ of $H_n$ is equal to
$$
\Psi(x^{(1)}\ldots x^{(r)}),\;\;\; x^{(1)}\in S_{|\lambda^{\zeta_{\tau_1}}|}(G),
\ldots, x^{(r)}\in S_{|\lambda^{\zeta_{\tau_\tau}}|}(G).
$$
Here the element $x^{(1)}\ldots x^{(r)}$ corresponds to $x$ under the group isomorphism between
$H_n$ and $S_{|\lambda^{\zeta_{\tau_1}}|}(G)\times\ldots\times S_{|\lambda^{\zeta_{\tau_\tau}}|}(G)$.

From the very definition of the function $\Psi$, see equation (\ref{PSI}), it follows that
\begin{equation}\label{PT1}
\Psi(x^{(1)}\ldots x^{(r)})=\Psi(x^{(1)})\ldots\Psi(x^{(r)}).
\end{equation}
Taking into account the formulae for the character $\chi^{\varrho^{\eta}\otimes\zeta^{\eta}}$ (see equations
(\ref{chchi5.3.1}), (\ref{chchi5.3.2})), and writing
$$
x^{(1)}=\left(d^{(1)},s^{(1)}\right),\ldots,x^{(r)}=\left(d^{(r)},s^{(r)}\right),
$$
we obtain
\begin{equation}
\chi^{\varrho^{\eta}\otimes\xi^{\eta}}\left(x^{(1)}\ldots x^{(r)}\right)=\chi^{\varrho^{\xi_{\tau_1}}}(x^{(1)})
\chi^{\lambda^{\zeta_{\tau_1}}}(s^{(1)})\ldots
\chi^{\varrho^{\xi_{\tau_r}}}(x^{(r)})
\chi^{\lambda^{\zeta_{\tau_r}}}(s^{(r)}).
\end{equation}
Now we use the Frobenius reciprocity, and represent the inner
$\left<\Psi,\chi^{\Lambda}\right>_{S_n(G)}$ as  a product,
\begin{equation}\label{PT3}
\left<\Psi,\chi^{\Lambda}\right>_{S_n(G)}=\left<\Psi(.),\chi^{\varrho^{\zeta_{\tau_1}}}(.)\chi^{\lambda^{\zeta_{\tau_1}}}(.)\right>_{S_{|\lambda^{\zeta_{\tau_1}}|}(G)}\ldots
\left<\Psi(.),\chi^{\varrho^{\zeta_{\tau_r}}}(.)\chi^{\lambda^{\zeta_{\tau_r}}}(.)\right>_{S_{|\lambda^{\zeta_r}|}(G)},
\end{equation}
where we have agreed that
$$
\chi^{\lambda^{\zeta_{\tau_j}}}(x^{(j)})=\chi^{\lambda^{\zeta_{\tau_j}}}\left(
\left(d^{(j)},s^{(j)}\right)\right)=\chi^{\lambda^{\zeta_{\tau_j}}}\left(s^{(j)}\right),\;\; s^{(j)}\in S_{|\lambda^{\zeta_{\tau_j}}|},\;\; j=1,\ldots,r.
$$
Now assume that $\zeta\in\widehat{G}$, $|\lambda^{\zeta}|=n$, and
$\chi^{\varrho^{\zeta}}$ is the character of the representation
$\left(\varrho^{\zeta},\left(V^{\zeta}\right)^{\otimes |\lambda^{\zeta}|}\right)$ of the group
$S_{|\lambda^{\zeta}|}(G)$ defined by (\ref{HHH3.5.1}). Then
we can write
\begin{equation}\label{PT4}
\begin{split}
&\left<\Psi(.),\chi^{\varrho^{\zeta}}(.)\chi^{\lambda^{\zeta}}(.)\right>_{S_{n}(G)}=\left<J_{\zeta}(.),\chi^{\lambda^{\zeta}}(.)\right>_{S_n},
\end{split}
\end{equation}
where
\begin{equation}\label{PT5}
J_{\zeta}(s)=\int\limits_{G}\ldots\int\limits_{G}
\Psi\left[\left((g_1,\ldots,g_n),s\right)\right]
\overline{\chi^{\varrho^{\zeta}}\left[\left((g_1,\ldots,g_n),s\right)\right]}
d\mu_G(g_1)\ldots d\mu_{G}(g_n),
\end{equation}
and where $s\in S_n$.
\begin{prop}\label{PropPT6}
Assume that for each $\zeta\in\widehat{G}$, and each $r=1,2,\ldots $ the functions $\widehat{p}_r(\zeta)$ are defined by equation (\ref{ch1.3.3}). Then $J_{\zeta}(s)$ defined by equation (\ref{PT5}) is given by
\begin{equation}\label{PT7}
J_{\zeta}(s)=\widehat{p}_{\mu(s)}(\zeta),\;\; s\in S_n,
\end{equation}
where the Young diagram $\mu(s)$ determines the cyclic structure of $s$, and
$\widehat{p}_{\mu(s)}(\zeta)$ is the Newton power symmetric function in $\widehat{\SYM}(\zeta)$
parameterized by $\mu(s)$.
\end{prop}
\begin{proof}
Recall that the characters $\chi^{\varrho^{\zeta}}$ are given by Proposition \ref{chchiProposition5.1.5}. First, assume that $s=(i_1i_2\ldots i_n)$. From equation
(\ref{chchi5.1.3}) we obtain
\begin{equation}\label{PT7.1}
\chi^{\varrho^{\zeta}}\left[\left((g_1,\ldots,g_n),(i_1i_2\ldots i_n)\right)\right]
=\chi^{\pi^{\zeta}}\left(g_{i_n}\ldots g_{i_2}g_{i_1}\right).
\end{equation}
The conjugacy class of $\left((g_1,\ldots,g_n),(i_1i_2\ldots i_n)\right)$ in $S_n(G)$
is determined by the map $\varrho: [G]\longrightarrow\Y$ defined by
$$
\varrho(c)=\left\{
  \begin{array}{ll}
    (n), & \hbox{$c$ is the conjugacy class of $g_{i_n}\ldots g_{i_2}g_{i_1}$}, \\
    \emptyset, & \hbox{otherwise}
  \end{array}
\right.
$$
This gives
\begin{equation}\label{PT.8}
\Psi\left[\left((g_1,\ldots,g_n),(i_1i_2\ldots i_n)\right)\right]=p_n\left(y_1(g_{i_n}\ldots g_{i_2}g_{i_1}),
y_2(g_{i_n}\ldots g_{i_2}g_{i_1}),\ldots\right).
\end{equation}
Inserting (\ref{PT7.1}), (\ref{PT.8}) into (\ref{PT5}), and
using the invariance of the Haar measure under the shifts by elements of the group, we find
\begin{equation}
\begin{split}
J_{\zeta}\left((i_1i_2\ldots i_n)\right)=\left<p_n\left(y_1(.),y_2(.),\ldots\right),
\chi^{\pi^{\zeta}}(.)\right>_{G}
=\widehat{p}_n(\zeta),
\end{split}
\end{equation}
so equation (\Ref{PT7}) holds true for $s=(i_1i_2\ldots i_n)$.

Next, assume that $s\in S_n$ is given by
$$
s=(i_1i_2\ldots i_l)(j_1j_2\ldots j_k),\;\; k+l=n.
$$
Then (\ref{chchi5.1.3}) enables us to write
\begin{equation}\label{PT9}
\chi^{\varrho^{\zeta}}\left[\left((g_1,\ldots,g_n),(i_1i_2\ldots i_l)(j_1j_2\ldots j_k)\right)\right]=
\chi^{\pi^{\zeta}}(g_{i_l}\ldots g_{i_2}g_{i_1})\chi^{\pi^{\zeta}}(g_{j_k}\ldots g_{j_2}g_{j_1}).
\end{equation}
Denote by $c_1$ the conjugacy class of $g_{i_l}\ldots g_{i_1}g_{i_1}$ in $G$, and
by $c_2$ the conjugacy class of $g_{j_k}\ldots g_{j_2}g_{j_1}$ in $G$. Then the function
$\varrho: [G]\rightarrow \Y$ describing the conjugacy class of
$\left((g_1,\ldots,g_n), (i_1i_2\ldots i_l)(j_1j_2\ldots j_k)\right)$ in $S_n(G)$
can be represented as
$$
\varrho(c)=\left\{
  \begin{array}{ll}
    (l), & c=c_1, c_1\neq c_2, \\
    (k), & c=c_2,c_1\neq c_2, \\
    (l)\cup (k), & c=c_1, c_1=c_2, \\
    \emptyset, & \hbox{otherwise.}
  \end{array}
\right.
$$
This implies the following formula
\begin{equation}\label{PT10}
\begin{split}
&\Psi\left[\left((g_1,\ldots,g_n),(i_1i_2\ldots i_l)(j_1j_2\ldots j_k)\right)\right]\\
&=p_{(l)}\left(y_1(c_1),y_2(c_1),\ldots\right)p_{(k)}\left(y_1(c_2),y_2(c_2),\ldots\right).
\end{split}
\end{equation}
Inserting (\ref{PT9}) and (\ref{PT10}) into (\ref{PT5}) we find
\begin{equation}
\begin{split}
&J_{\zeta}\left((i_1i_2\ldots i_l)(j_1j_2\ldots j_k)\right)\\
&=\left<p_l\left(y_1(.),y_2(.),\ldots\right),
\chi^{\pi^{\zeta}}(.)\right>_{G}\left<p_k\left(y_1(.),y_2(.),\ldots\right),
\chi^{\pi^{\zeta}}(.)\right>_{G}\\
&=\widehat{p}_l(\zeta)\widehat{p}_k(\zeta),
\end{split}
\end{equation}
so (\ref{PT7}) holds true for $s=(i_1i_2\ldots i_l)(j_1j_2\ldots j_k)$ as well. The general case where
$s\in S_n$ is a product of an arbitrary number of cycles is considered in the same way.
\end{proof}
Taking into account  Proposition \ref{PropPT6} we see that equations (\ref{PT3}), (\ref{PT4}) give us
\begin{equation}
\begin{split}
&\left<\Psi,\chi^{\Lambda}\right>_{S_n(G)}=
\left<J_{\zeta_{\tau_1}},\chi^{\lambda^{\zeta_{\tau_1}}}\right>_{S_{|\lambda^{\zeta_{\tau_1}}|}}
\ldots\left<J_{\zeta_{\tau_r}},\chi^{\lambda^{\zeta_{\tau_r}}}\right>_{S_{|\lambda^{\zeta_{\tau_r}}|}}\\
&=\left<\widehat{p}_{\mu(.)}(\zeta_{\tau_1}),\chi^{\lambda^{\zeta_{\tau_1}}}(.)\right>_{S_{|\lambda^{\zeta_{\tau_1}}|}}
\ldots\left<\widehat{p}_{\mu(.)}(\zeta_{\tau_r}),\chi^{\lambda^{\zeta_{\tau_r}}}(.)\right>_{S_{|\lambda^{\zeta_{\tau_r}}|}}\\
&=\widehat{s}_{\lambda^{\zeta_{\tau_1}}}(\zeta_{\tau_1})\ldots\widehat{s}_{\lambda^{\zeta_{\tau_r}}}(\zeta_{\tau_r})
=\prod\limits_{\zeta\in\widehat{G}}\widehat{s}_{\Lambda(\zeta)}(\zeta).
\end{split}
\end{equation}
Theorem \ref{Theoremch1.3.2} is proved. \qed
\section{Harmonic functions defined by the $z$-measures $M_z^{(n)}$}\label{SECTIONHF1}
\subsection{The branching graph $\Y(\widehat{G})$}
Set
\begin{equation}\label{4.1.2}
\Y(\widehat{G})=\bigsqcup\limits_{n=0}^{\infty}\Y_n(\widehat{G}),
\end{equation}
where $\Y_n(\widehat{G})$ is defined by (\ref{YnG}).
The set $\Y_0(\widehat{G})$ is, by definition, the empty set. Note that $\Y_1(\widehat{G})$ can be identified
with $\widehat{G}$. The canonical inclusion $i_{n,k}:\; S_k(G)\rightarrow S_n(G)$ is defined as
$i_{n,k}(g,s)=(\tilde{g},\tilde{s})$, where
$$
\tilde{g}=(g,e_{G},\ldots,e_{G})\in G^k\times G^{n-k}=G^n,
$$
and
$$
\tilde{s}=s(k+1)(k+2)\ldots(n)\in S_n.
$$
Under the inclusion $i_{n,k}$, $S_k(G)$ is a subgroup of $S_n(G)$.

Let us introduce the following notation. If $\Lambda\in\Y_n(\widehat{G})$ is obtained by adding a box to one of the Young diagrams from the family $M\in\Y_{n-1}(\widehat{G})$, we write
$M\nearrow\Lambda$, and say that the family $M$ is adjacent to the family $\Lambda$.

Assume that $M\in\Y_{n-1}(\widehat{G})$, $\Lambda\in\Y_n(\widehat{G})$, and $M\nearrow\Lambda$.
Then there exists $\zeta_{M,\Lambda}\in\widehat{G}$ such that the Young diagram $\Lambda(\zeta_{M,\Lambda})$ is obtained
from the Young diagram $M\left(\zeta_{M,\Lambda}\right)$ by adding a box. Thus $\zeta_{M,\Lambda}\in\widehat{G}$
is uniquely determined by $M,\Lambda$ such that $M\nearrow\Lambda$.
\begin{thm}\label{THEOREM4.5.1}
For $\Lambda\in\Y_n(\widehat{G})$, and the corresponding irreducible unitary representation $\pi^{\Lambda}$ of $S_n(G)$,
the restriction of $\pi^{\Lambda}$ to subgroup $S_{n-1}(G)$ has the following decomposition into irreducible unitary
representations
\begin{equation}\label{4.5.2}
\Res_{S_{n-1}(G)}^{S_n(G)}\pi^{\Lambda}\cong\underset{M\in\Y_{n-1}(\widehat{G}):\; M\nearrow\lambda}{\bigoplus}\dim\zeta_{M,\Lambda}\;\pi^{M}.
\end{equation}
\end{thm}
\begin{proof} See  Hora, Hirai, Hirai \cite{HoraHiraiHiraiII}, Section 1.3.
\end{proof}
\begin{defn}We consider  $\Y(\widehat{G})$ defined by (\ref{4.1.2}) as a branching graph for the wreath
products $S_0(G)$, $S_1(G)$, $S_2(G)$, $\ldots$. The branching rule (\ref{4.5.2}) gives the edge structure of $\Y(\widehat{G})$. That is, $M\in\Y_{n-1}(\widehat{G})$ and $\Lambda\in\Y_n(\widehat{G})$ are joined by an edge if and only if $M\nearrow\Lambda$. The multiplicity of the corresponding edge is given by $\dim\zeta_{M,\Lambda}$.
\end{defn}
\begin{defn} A function $\varphi: \Y(\widehat{G})\rightarrow\C$ is said to be harmonic if
\begin{equation}\label{4.7.2}
\varphi(M)=\sum\limits_{\Lambda:\; \Lambda\searrow M}\dim\zeta_{\Lambda,M}\;\varphi(\Lambda),
\end{equation}
and $\varphi(\emptyset)=1$.
\end{defn}
\subsection{Dimensions of irreducible representations}
The dimensions of the irreducible representations of $S_n(G)$ are given by the following Proposition.
\begin{prop}Assume that $G$ is a compact group such that $\widehat{G}$ is at most countable set.
Let $\Lambda:\;\widehat{G}\longrightarrow\Y$ be the map that defines an irreducible representation of $S_n(G)$.
We have
\begin{equation}\label{6.1.2}
\DIM\Lambda:=\chi^{\Lambda}\left(e_{S_n(G)}\right)
=n!\prod\limits_{\zeta\in\widehat{G}}\left(\dim\zeta\right)^{|\Lambda(\zeta)|}\frac{\dim\Lambda(\zeta)}{|\Lambda(\zeta)|!},
\end{equation}
where $\chi^{\Lambda}$ is the character of the irreducible representation of $S_n(G)$ parameterized by $\Lambda$,
and $e_{S_n(G)}$ is the unit element of $S_n(G)$.
\end{prop}
\begin{proof}In the case where $G$ is a finite group formula (\ref{6.1.2}) is well-known, see Macdonald \cite{Macdonald},
Appendix B, \S 9. In the case where $G$ is a compact non-finite group, formula (\ref{6.1.2}) appears
in Hora, Hirai, Hirai \cite{HoraHiraiHiraiII}, Section 1.2. The fact that for finite and for compact non-finite groups
$G$ the formula for $\DIM\Lambda$ has the same form follows from the interpretation of $\DIM\Lambda$
as the number of paths on the branching graph $\Y\left(\widehat{G}\right)$  from $\emptyset$ to $\Lambda$.
\end{proof}
\subsection{Harmonic functions $\varphi_z$}
Define $\varphi_z:\; \Y(\widehat{G})\rightarrow\R_{\geq 0}$ by
\begin{equation}\label{6.5.2}
\varphi_z(\Lambda)=\frac{M_z^{(n)}(\Lambda)}{\DIM\Lambda},
\end{equation}
where $M_z^{(n)}$ is given by equation (\ref{13.2.2.2}).
\begin{thm}\label{THEOREM6.5.2}The function $\varphi_z$ is harmonic on $\Y(\widehat{G})$.
\end{thm}
\begin{proof} We need to check that
\begin{equation}\label{6.5.3}
\sum\limits_{\Lambda:\; \Lambda\searrow M}
\dim(\zeta_{M,\Lambda})\;\varphi_z(\Lambda)=\varphi_z(M)
\end{equation}
is satisfied for all $M\in\Y_{n-1}(\widehat{G})$. We can write $\varphi_z(\Lambda)$ explicitly as
\begin{equation}\label{6.5.3.1}
\varphi_z(\Lambda)=\frac{1}{(I)_n}
\prod\limits_{\zeta\in\widehat{G}}\frac{(\alpha(\zeta)
\overline{\alpha(\zeta)})_{|\Lambda(\zeta)|}}{\left(\dim\zeta\right)^{|\Lambda(\zeta)|}}
\frac{\mathfrak{m}_{\alpha(\zeta)}^{(|\Lambda(\zeta)|)}(\Lambda(\zeta))}{\dim\Lambda(\zeta)},
\end{equation}
as it follows from (\ref{6.2.1}), (\ref{6.1.2}), and (\ref{6.5.2}). Let $M:\; \widehat{G}\longrightarrow \Y$
be a fixed element of $\Y_{n-1}(\widehat{G})$. Each map $\Lambda$, $\Lambda\in\Y_n(\widehat{G})$, such that $\Lambda\searrow M$
can be defined by \\
(a) $\Lambda(\zeta)=M(\zeta)$, $\zeta\in\widehat{G}$, $\zeta\neq \zeta_{M,\Lambda}$;\\
and by\\
(b) $\Lambda(\zeta_{M,\Lambda})$ is such that the Young diagram $\Lambda(\zeta_{M,\Lambda})$ is obtained from the Young diagram  $M(\zeta_{M,\Lambda})$ by adding one box.\\
Now we have
\begin{equation}\label{6.5.4}
\begin{split}
&\sum\limits_{\Lambda:\;\Lambda\searrow M}\dim(\zeta_{M,\Lambda})\varphi_z(\Lambda)\\
&=\frac{1}{(I)_n}
\sum\limits_{\zeta_{M,\Lambda}\in\widehat{G}}\;\;
\sum\limits_{\Lambda(\zeta_{M,\Lambda})\searrow\; M(\zeta_{M,\Lambda})}
\frac{\mathfrak{m}_{\alpha(\zeta_{M,\Lambda})}^{(|\Lambda(\zeta_{M,\Lambda})|)}(\Lambda(\zeta_{M,\Lambda}))}{\dim
\Lambda(\zeta_{M,\Lambda})}
\left(\prod\limits_{\zeta\in\widehat{G}}\frac{1}{(\dim\zeta)^{|M(\zeta)|}}\right)\\
&\times\underset{\zeta\neq\zeta_{M,\Lambda}}{\prod\limits_{\zeta\in\widehat{G}}}
\frac{\mathfrak{m}_{\alpha(\zeta)}^{|M(\zeta)|)}(M(\zeta))}{\dim
M(\zeta)}
\underset{\zeta\neq\zeta_{M,\Lambda}}{\prod\limits_{\zeta\in\widehat{G}}}(\alpha(\zeta)\overline{\alpha(\zeta)})_{|M(\zeta)|}.
\end{split}
\end{equation}
Since
$$
\hat{\varphi}_{\alpha}(\lambda)=\frac{\mathfrak{m}_{\alpha}^{(|\lambda|)}(\lambda)}{\dim\lambda}
$$
is a harmonic function on $\Y$ (see, for example, Borodin and Olshanski \cite{BorodinOlshanskiMarkov}), we have
\begin{equation}\label{6.5.5}
\sum\limits_{\Lambda(\zeta_{M,\Lambda})\searrow\; M(\zeta_{M,\Lambda})}\frac{\mathfrak{m}_{\alpha(\zeta_{M,\Lambda})}^{(|\Lambda(\zeta_{M,\Lambda})|)}(\Lambda(\zeta_{M,\Lambda}))}{\dim
\Lambda(\zeta_{M,\Lambda})}
=\frac{\mathfrak{m}_{\alpha(\zeta_{M,\Lambda})}^{(|M(\zeta_{M,\Lambda})|)}(M(\zeta_{M,\Lambda}))}{\dim
M(\zeta_{M,\Lambda})}.
\end{equation}
In addition, note that
$$
\left(\alpha(\zeta_{M,\Lambda})\overline{\alpha(\zeta_{M,\Lambda})}\right)_{|\Lambda(\zeta_{M,\Lambda})|}
=\left(\alpha(\zeta_{M,\Lambda})\overline{\alpha(\zeta_{M,\Lambda})}\right)_{|M(\zeta_{M,\Lambda})|}
\left(\alpha(\zeta_{M,\Lambda})\overline{\alpha(\zeta_{M,\Lambda})}+|M(\zeta_{M,\Lambda})|\right).
$$
The Parseval identity (equation (\ref{Parseval})) implies
\begin{equation}\label{6.5.6}
\sum\limits_{\zeta_{M,\Lambda}\in\widehat{G}}\left(\alpha(\zeta_{M,\Lambda})\overline{\alpha(\zeta_{M,\Lambda})}+|M(\zeta_{M,\Lambda})|\right)
=I+n-1.
\end{equation}
It follows from (\ref{6.5.3.1}), (\ref{6.5.4}), (\ref{6.5.5}) and (\ref{6.5.6}) that equation (\ref{6.5.3}) holds true.
\end{proof}
\section{The spectral $z$-measures}\label{SECTIONSPECTRAL}
The fact that the $z$-measures $M_z^{(n)}$ can be understood in terms of
harmonic functions $\varphi_z$ on $\Y(G)$ enables us to obtain an integral representation
for $M_z^{(n)}$. This representation will follow from Theorem \ref{THEOREM7.1.1} below.
\begin{thm}\label{THEOREM7.1.1}
Any harmonic function on $\Y(\widehat{G})$ admits the following representation
\begin{equation}\label{7.1.2}
\varphi(\Lambda)=\int\limits_{\triangle}\mathbb{K}(\Lambda,\omega)dP(\omega).
\end{equation}
Here,
the probability measure $P$ on $\triangle$ is uniquely determined by $\varphi$,
and
\begin{equation}\label{7.1.3}
\begin{split}
&\triangle=\biggl\{(\alpha,\beta,\delta)|\;
\alpha=\left(\alpha_{\zeta,i}\right)_{\zeta\in\widehat{G},\; i\in\mathbb{N}},\;\;
\beta=\left(\beta_{\zeta,i}\right)_{\zeta\in\widehat{G},\; i\in\mathbb{N}},\;\;
\delta=(\delta_{\zeta})_{\zeta\in\widehat{G}};\\
&\alpha_{\zeta,1}\geq\alpha_{\zeta,2}\geq \ldots\geq 0;\;\;
\beta_{\zeta,1}\geq\beta_{\zeta,2}\geq\ldots\geq 0;\;\;\delta_{\zeta}\geq 0;\\
&\sum\limits_{i=1}^{\infty}\left(\alpha_{\zeta,i}+\beta_{\zeta,i}\right)\leq\delta_{\zeta},\;
\forall\zeta\in\widehat{G};\;\;\sum\limits_{\zeta\in\widehat{G}}\delta_{\zeta}=1\biggr\}
\end{split}
\end{equation}
is the generalized Thoma set. The kernel $\mathbb{K}(\Lambda,\omega)$ is defined by
\begin{equation}\label{7.1.4}
\mathbb{K}(\Lambda,\omega)=\prod\limits_{\zeta\in\widehat{G}}
\frac{1}{(\dim\zeta)^{|\Lambda(\zeta)|}}s^{\zeta}_{\Lambda(\zeta)}(\omega),
\end{equation}
 where
\begin{equation}\label{7.1.5}
s^{\zeta}_{\Lambda(\zeta)}(\omega)=\sum\limits_{\varrho\in\Y_{\left|\Lambda(\zeta)\right|}}
\frac{1}{z_{\varrho}}\chi^{\Lambda(\zeta)}_{\varrho}p_{\varrho}^{\zeta}(\omega),
\end{equation}
The supersymmetric power sums $p_{\varrho}^{\zeta}$ in equation (\ref{7.1.5}) are defined by
\begin{equation}\label{7.1.6}
p_{\varrho}^{\zeta}(\omega)=p_{\varrho_1}^{\zeta}(\omega)\ldots p_{\varrho_l}^{\zeta}(\omega),
\;\;\;
\varrho=\left(\varrho_1\geq\ldots\geq\varrho_l>0\right)\in\Y,
\end{equation}
and by
\begin{equation}\label{7.1.7}
p_k^{\zeta}(\omega)=p_k^{\zeta}(\alpha,\beta,\delta)
=\left\{
  \begin{array}{ll}
    \sum\limits_{i=1}^{\infty}\left(\alpha_{\zeta,i}^k+(-1)^{k-1}\beta_{\zeta,i}^k\right), & k\geq 2, \\
    \delta_{\zeta}, &  k=1.
  \end{array}
\right.
\end{equation}
\end{thm}
\begin{proof}See Hora and Hirai \cite{HoraHirai}, Theorem 3.1.
\end{proof}
\begin{cor}For each $n=1,2,\ldots $ we have
\begin{equation}\label{7.2.2}
\varphi_z(\Lambda)=\frac{M_z^{(n)}(\Lambda)}{\DIM\Lambda}
=\int\limits_{\triangle}\mathbb{K}(\Lambda,\omega)dP_{z}(\omega), \;\Lambda\in\Y_n(G),
\end{equation}
where $\left(M_z^{(n)}\right)_{n=1}^{\infty}$ are the $z$-measures for wreath products $\left(S_n\left(G\right)\right)_{n=1}^{\infty}$
defined by equation (\ref{13.1.2}), and given explicitly by Theorem \ref{Theorem13.2.2}.
In equation (\ref{7.2.2}), $P_z$ is a probability measure on $\triangle$ corresponding to $\left(M_z^{(n)}\right)_{n=1}^{\infty}$.
\end{cor}
From equations (\ref{7.2.2}) and (\ref{13.1.2}) we obtain formula
(\ref{ZcharacterRepresentation}), which is an integral representation of the character
$\chi_z$ of the generalized regular representation $\left(T_z,L^2\left(\mathfrak{S}_G,
\mu_{\mathfrak{S}_G}^{\Ewens}\right)\right)$ of $S_{\infty}(G)\times S_{\infty}(G)$.
Equation (\ref{ZcharacterRepresentation}) can be understood as a spectral decomposition
of $\chi_z$.
Equation (\ref{7.1.3}) gives the explicit expression for $\triangle$. The formula for the kernel
$f_{\omega}(\varrho)$ in (\ref{ZcharacterRepresentation}) is known;
see Hora and Hirai \cite{HoraHirai}, Theorem 3.14. The problem we address in the following is to describe the measure $P_z$.

Let $\mathfrak{m}^{(n)}_{z}$ be the $z$-measure on $\Y_n$.  It is known that
the $z$-measures $\mathfrak{m}^{(n)}_{z}$ form a coherent system of probability
measures on the Young graph $\Y$, see, for example, Borodin and Olshanski \cite{BorodinOlshanskiMarkov}. This gives the following integral representation
\begin{equation}\label{13.2.2.new}
\frac{\mathfrak{m}^{(n)}_{z}(\lambda)}{\dim\lambda}
=\int\limits_{\Omega_{0}}\overset{\circ}{s}_{\lambda}(\omega)d\overset{\circ}{P}_z(\omega),
\end{equation}
where
\begin{equation}
\Omega_{0}=\left\{\omega=(\alpha,\beta);\;
\begin{array}{c}
  \alpha=(\alpha_1\geq\alpha_2\geq\ldots\geq 0) \\
  \beta=(\beta_1\geq\beta_2\geq\ldots\geq 0)
\end{array},
\;
\sum\limits_{i=1}^{\infty}(\alpha_i+\beta_i)=1\right\}.
\end{equation}
Here
\begin{equation}
\overset{\circ}{s}_{\lambda}(\omega)=\sum\limits_{\varrho\in\Y_n}
\frac{1}{z_{\varrho}}\chi^{\lambda}_{\varrho}\;\overset{\circ}{p}_{\varrho}(\omega),
\end{equation}
and the supersymmetric power sums $\overset{\circ}{p}_{\varrho}$  are defined by
\begin{equation}
\overset{\circ}{p}_{\varrho}(\omega)=\overset{\circ}{p}_{\varrho_1}(\omega)\ldots \overset{\circ}{p}_{\varrho_l}(\omega),
\;\;\;
\varrho=\left(\varrho_1\geq\ldots\geq\varrho_l>0\right)\in\Y,
\end{equation}
and by
\begin{equation}
\overset{\circ}{p}_k(\omega)=\overset{\circ}{p}_k(\alpha,\beta)
=\left\{
  \begin{array}{ll}
    \sum\limits_{i=1}^{\infty}\left(\alpha_{i}^k+(-1)^{k-1}\beta_{i}^k\right), & k\geq 2, \\
    1, &  k=1.
  \end{array}
\right.
\end{equation}
The measure $\overset{\circ}{P}_z$ is a probability measure on $\Omega_0$ called the spectral $z$-measure for $\left(\mathfrak{m}_z^{(n)}\right)_{n=1}^{\infty}$ and is associated with the generalized regular representation of the infinite bi-symmetric group
$S(\infty)\times S(\infty)$. If we consider $\left(\Omega_0,\overset{\circ}{P}_z\right)$
as a probability space, then $\alpha_i$, $\beta_i$ are random variables whose distribution is determined by $\overset{\circ}{P}_z$. The random variables
$\alpha_i$, $\beta_i$ were studied in detail by Borodin and Olshanski,
and the distribution of $\alpha_i$, $\beta_i$ is described explicitly in
Refs. \cite{Borodin1,Borodin2,BorodinOlshanskiLetters}.

Now, consider the probability space $(\triangle, P_z)$, where $\triangle$ is the generalized Thoma set defined by (\ref{7.1.3}), and $P_z$ is the probability measure
in the spectral decomposition of $\chi_z$, equation (\ref{ZcharacterRepresentation}).
Denote by $\tilde{\alpha}_{\zeta,i}$, $\tilde{\beta}_{\zeta,i}$, and $\tilde{\delta}_{\zeta}$ the coordinates of $\triangle$. These coordinates can be
viewed as random variables. The problem is to describe the distribution of $\tilde{\alpha}_{\zeta,i}$, $\tilde{\beta}_{\zeta,i}$, and $\tilde{\delta}_{\zeta}$,
and thus to describe $P_z$. In Theorem \ref{THEOREMSPECTRAL} we express
$\tilde{\alpha}_{\zeta,i}$, $\tilde{\beta}_{\zeta,i}$, and $\tilde{\delta}_{\zeta}$
in terms of random variables with known distribution.
\begin{thm}\label{THEOREMSPECTRAL}
Set
\begin{equation}
\nabla_{\widehat{G}}=\left\{\delta:\;\widehat{G}\longrightarrow [0,1],\;\;\sum\limits_{\zeta\in\widehat{G}}\delta(\zeta)=1\right\}.
\end{equation}
Assume that $\tau(\zeta)=\alpha(\zeta)\overline{\alpha}(\zeta),\;\zeta\in \widehat{G}$  (where $\alpha(\zeta)$ is defined by equation (\ref{13.2.2.3})) is such that on $\nabla_{\widehat{G}}$  there exists a probability distribution $D(\tau)$ with the density
$$
\frac{\Gamma(\sum\limits_{\zeta\in\widehat{G}}\tau(\zeta))}{\prod\limits_{\zeta\in\widehat{G}}\Gamma(\tau(\zeta))}
\prod\limits_{\zeta\in\widehat{G}}\left(\delta(\zeta)\right)^{\tau(\zeta)-1}
$$
with respect to the Lebesgue measure on $\nabla_{\widehat{G}}$.
In addition, assume that for each $\zeta$, $\zeta\in\widehat{G}$, the joint distribution of the random variables
$$
\alpha_{\zeta,1}\geq\alpha_{\zeta,2}\geq\ldots\geq 0,\;\;\beta_{\zeta,1}\geq\beta_{\zeta,2}\geq\ldots\geq 0
$$
is given by $\overset{\circ}{P}_{\alpha(\zeta)}$. Furthermore, assume that the joint distribution of
random variables $\left(\delta_\zeta\right)_{\zeta\in\widehat{G}}$ is given by $D(\tau)$.
Set
$$
\tilde{\alpha}_{\zeta}=\delta_{\zeta}\alpha_{\zeta},\;\;\tilde{\beta}_{\zeta}=\delta_{\zeta}\beta_{\zeta},\;\;\tilde{\delta}_{\zeta}=
\delta_{\zeta};\;\;\forall\zeta\in\widehat{G}.
$$
Then the joint distribution  of $(\tilde{\alpha}_{\zeta})_{\zeta\in\widehat{G}}$, $(\tilde{\beta}_{\zeta})_{\zeta\in\widehat{G}}$,
and $(\tilde{\delta}_{\zeta})_{\zeta\in\widehat{G}}$ is given by the spectral $z$-measure $P_z$.
\end{thm}
\begin{proof}
Set
$$
\overset{\circ}{\varphi}_{\alpha(\zeta)}(\Lambda(\zeta))=\frac{\mathfrak{m}_{\alpha(\zeta)}(\Lambda(\zeta))}{\dim\Lambda(\zeta)}.
$$
The harmonic function $\varphi_z$ (defined by equation (\ref{6.5.2})) can be written as
\begin{equation}
\varphi_z(\Lambda)=\frac{1}{(I)_n}
\prod\limits_{\zeta\in\widehat{G}}
\frac{(\alpha(\zeta)\overline{\alpha(\zeta)})_{|\Lambda(\zeta)|}}{(\dim\zeta)^{|\Lambda(\zeta)|}}
\overset{\circ}{\varphi}_{\alpha(\zeta)}(\Lambda(\zeta)),
\end{equation}
as it follows from (\ref{6.5.3.1}). Using (\ref{13.2.2.new}) we find
\begin{equation}\label{7.4.3}
\varphi_z(\Lambda)=
\frac{1}{(I)_n}
\prod\limits_{\zeta\in\widehat{G}}\left[\int\limits_{\Omega_0(\zeta)}
\frac{(\alpha(\zeta)\overline{\alpha(\zeta)})_{|\Lambda(\zeta)|}}{(\dim\zeta)^{|\Lambda(\zeta)|}}
\;\overset{\circ}{s}_{\Lambda(\zeta)}(\omega_{\zeta})d\overset{\circ}{P}_{\alpha(\zeta)}(\omega_{\zeta})\right],
\end{equation}
where
\begin{equation}
\Omega_{0}(\zeta)=\left\{\omega_{\zeta}=(\alpha_{\zeta},\beta_{\zeta});\;
\begin{array}{c}
  \alpha_{\zeta}=(\alpha_{\zeta,1}\geq\alpha_{\zeta,2}\geq\ldots\geq 0) \\
  \beta_{\zeta}=(\beta_{\zeta,1}\geq\beta_{\zeta,2}\geq\ldots\geq 0)
\end{array},
\;
\sum\limits_{i=1}^{\infty}(\alpha_{\zeta,i}+\beta_{\zeta,i})=1\right\}.
\nonumber
\end{equation}
Under the assumptions in the statement of Theorem \ref{THEOREMSPECTRAL} we can write
\begin{equation}\label{7.7.1}
\frac{1}{(I)_n}
\prod\limits_{\zeta\in\widehat{G}}
(\alpha(\zeta)\overline{\alpha(\zeta)})_{|\Lambda(\zeta)|}
=\underset{\nabla_{\widehat{G}}}{\int}\prod\limits_{\zeta\in\widehat{G}}\left(\delta(\zeta)\right)^{|\Lambda(\zeta)|}
D(\tau)\left(\prod\limits_{\zeta\in\widehat{G}}d\delta(\zeta)\right),
\end{equation}
where $\tau(\zeta)=\alpha(\zeta)\overline{\alpha(\zeta)}$, $\zeta\in\widehat{G}$. Equations (\ref{7.1.4}), (\ref{7.2.2}), (\ref{7.4.3}), and (\ref{7.7.1}) together with the fact that the Schur functions  are homogeneous give us the following condition
\begin{equation}\label{7.8.1}
\begin{split}
&\int\limits_{\nabla}\left[\prod\limits_{\zeta\in\widehat{G}}\;\;\int\limits_{\Omega_0(\zeta)}
\;\overset{\circ}{s}_{\Lambda(\zeta)}(\delta(\zeta)\omega_{\zeta})d\overset{\circ}{P}_{\alpha(\zeta)}(\omega_{\zeta})\right]D(\tau)\left(\prod\limits_{\zeta\in\widehat{G}}d\delta(\zeta)\right)\\
&=\int\limits_{\triangle}\left(\prod\limits_{\zeta\in\widehat{G}}\overset{\circ}{s}_{\Lambda(\zeta)}
(\tilde{\omega}_{\zeta})\right)dP_z(\tilde{\omega}_{\zeta}).
\end{split}
\end{equation}
Now, the statement of Theorem \ref{THEOREMSPECTRAL} follows from equation (\ref{7.8.1}), and from the fact that
$P_z$ is a unique probability measure on $\triangle$, for which equation (\ref{7.2.2}) is satisfied.
\end{proof}


\begin{thebibliography}{99}
\bibitem{Borodin1}
 Borodin, A. M. Characters of symmetric groups, and correlation functions of point processes. (Russian) Funktsional. Anal. i Prilozhen. 34 (2000), no. 1, 12–-28, 96; translation in Funct. Anal. Appl. 34 (2000), no. 1, 10–-23.
 \bibitem{Borodin2}
  Borodin, A. M. Harmonic analysis on the infinite symmetric group, and the Whittaker kernel. (Russian) Algebra i Analiz 12 (2000), no. 5, 28–-63; translation in St. Petersburg Math. J. 12 (2001), no. 5, 733
  -–759
\bibitem{BorodinOlshanskiLetters} Borodin, A.; Olshanski, G. Point processes and the infinite symmetric group. Math. Res. Lett. 5 (1998), 799--816.
\bibitem{BorodinOlshanskiRSK} Borodin, A.; Olshanski, G. z-measures on partitions, Robinson-Schensted-Knuth correspondence, and $\beta=2$ random matrix ensembles. Random matrix models and their applications, 71–-94, Math. Sci. Res. Inst. Publ., 40, Cambridge Univ. Press, Cambridge, 2001.
\bibitem{BorodinOlshanskiKernel} Borodin, A.; Olshanski, G. Distributions on partitions, point processes, and the hypergeometric kernel. Comm. Math. Phys. 211 (2000), no. 2, 335–-358.
\bibitem{BorodinOlshanskiUnitary} Borodin, A.; Olshanski, G. Harmonic analysis on the infinite-dimensional unitary group and determinantal point processes. Ann. of Math. (2) 161 (2005), no. 3, 1319–-1422.
\bibitem{BorodinOlshanskiMarkov}Borodin, A.; Olshanski, G. Markov processes on partitions. Probab. Theory Related Fields 135 (2006), no. 1, 84–-152.
\bibitem{BorodinOlshanskiBook}Borodin, A.; Olshanski, G. Representations of the infinite symmetric group. Cambridge Studies in Advanced Mathematics, 160. Cambridge University Press, Cambridge, 2017.
\bibitem{BorodinOlshanskiStrahov} Borodin, A.; Olshanski, G.; Strahov, E. Giambelli compatible point processes. Adv. in Appl. Math. 37 (2006), no. 2, 209–-248.
 \bibitem{CuencaOlshanski1}
 Cuenca, C.; Olshanski, G. Infinite-dimensional groups over finite fields and Hall-Littlewood symmetric functions. Adv. Math. 395 (2022), Paper No. 108087.
\bibitem{GorinKerovVershik} Gorin, V.; Kerov, S.; Vershik, A. Finite traces and representations of the group of infinite matrices over a finite field. Adv. Math. 254 (2014), 331–-395.
\bibitem{HiraiHiraiHoraI}
 Hirai, T.; Hirai, E.; Hora, A. Limits of characters of wreath products $\mathfrak{S}_n(T)$ of a compact group $T$ with the symmetric groups and characters of $\mathfrak{S}_{\infty}(T)$. I. Nagoya Math. J. 193 (2009), 1–-93.
 \bibitem{HoraLectures}
 Hora, A.  Lectures on "Introduction to Asymptotic Theory for Representations and Characters of Symmetric Groups" given at Wroclaw University in April -- June, 2007. Available online
 https://www.math.sci.hokudai.ac.jp/~hora/Wroclaw-LectureNote-ver27feb08.pdf.
\bibitem{HoraHiraiHiraiII} Hora, A.; Hirai, T.; Hirai, E. Limits of characters of wreath products $\mathfrak{S}_n(T)$ of a compact group $T$ with the symmetric groups and characters of
$\mathfrak{S}_{\infty}(T)$. II. From a viewpoint of probability theory. J. Math. Soc. Japan 60 (2008), no. 4, 1187–-1217.
\bibitem{HoraHirai}
 Hora, A.; Hirai, T. Harmonic functions on the branching graph associated with the infinite wreath product of a compact group. Kyoto J. Math. 54 (2014), no. 4, 775–-817.


Internat. Math. Res. Notices , no. 4, (1998), 173–-199.
 \bibitem{KerovDissertation} Kerov, S. V. Asymptotic representation theory of the symmetric group and its applications in analysis. Translated from the Russian manuscript by N. V. Tsilevich. With a foreword by A. Vershik and comments by G. Olshanski. Translations of Mathematical Monographs, 219. American Mathematical Society, Providence, RI, 2003.

\bibitem{KerovOlshanskiVershikAnnouncement}
Kerov, S.; Olshanski, G.; Vershik, A. Harmonic analysis on the infinite symmetric group. A deformation of the regular representation. C. R. Acad. Sci. Paris Sér. I Math. 316 (1993), no. 8, 773–-778.
\bibitem{KerovOlshanskiVershik}
Kerov, S.; Olshanski, G.; Vershik, A. Harmonic analysis on the infinite symmetric group. Invent. Math. 158 (2004), no. 3, 551–-642.
\bibitem{Kingman1} Kingman, J. F. C. Random partitions in population genetics. Proc. R. Soc. London A 361, (1978), 1--20.
\bibitem{Kingman2}Kingman, J. F. C. The representation of partition structures. J. London Math. Soc. 18 (1978), 374--380.
\bibitem{Macdonald} Macdonald, I. Symmetric Functions and Hall Polynomials. Oxford Mathematical Monographs. Oxford University Press, second edition, 1995.
\bibitem{OkounkovSchur}Okounkov, A.  Infinite wedge and random partitions.
Selecta Math. (N.S.) 7 (2001) 57--81.
 \bibitem{OlshanskiPointProcesses}Olshanski, G. Point processes related to the infinite symmetric group. The orbit method in geometry and physics (Marseille, 2000), 349–-393, Progr. Math., 213, Birkhäuser Boston, Boston, MA, 2003.
 \bibitem{OlshanskiUnitary} Olshanski, G. The problem of harmonic analysis on the infinite-dimensional unitary group. J. Funct. Anal. 205 (2003), no. 2, 464–-524.
 \bibitem{OlshanskiNato} Olshanski, G. An introduction to harmonic analysis on the infinite symmetric group. Asymptotic combinatorics with applications to mathematical physics (St. Petersburg, 2001), 127–-160, Lecture Notes in Math., 1815, Springer, Berlin, 2003.
\bibitem{StrahovMPS}
 Strahov, E. Multiple partition structures and harmonic functions on branching graphs.
Advances in Applied Mathematics 153 (2024), Article no. 102617.
 \bibitem{StrahovIsr25}
 Strahov, E. Generalized regular representations of big wreath products. Israel. J. Math. (2025)
 (in press).
 \bibitem{StrahovRefinement}
 Strahov, E. A refinement of the Ewens sampling formula. arXiv:2403.05077.
 \bibitem{Tavare} Tavar$\acute{\mbox{e}}$, S.  The magical Ewens sampling formula.
Bull. London Math. Soc. 53 (2021), 1563--1582.
 \bibitem{Thoma}
Thoma, E. Die unzerlegbaren, positiv-definiten Klassenfunktionen der abzählbar unendlichen, symmetrischen Gruppe. (German) Math. Z. 85 (1964), 40–-61.
\bibitem{VershikKerov1}
 Vershik, A. M.; Kerov, S. V. Asymptotic theory of the characters of a symmetric group. (Russian) Funktsional. Anal. i Prilozhen. 15 (1981), no. 4, 15–-27.
\end{thebibliography}
\end{document}